\documentclass[english]{article}
\usepackage{amsmath}
\usepackage{amssymb}
\usepackage{amsthm}
\usepackage[english]{babel}
\usepackage{times}
\usepackage[T1]{fontenc}
\usepackage[latin1]{inputenc}
\usepackage[all]{xy}
\usepackage{geometry}
\usepackage{color}

\newtheorem{teo}{Theorem}[section]
\newtheorem{coro}[teo]{Corollary}
\newtheorem{defn}[teo]{Definition}
\newtheorem{lem}[teo]{Lemma}
\newtheorem{pro}[teo]{Proposition}
\theoremstyle{definition}
\newtheorem{oss}[teo]{Remark}

\newcommand{\D}{\mathbb{D}}
\newcommand{\B}{\mathbb{B}}
\newcommand{\HH}{\mathbb{H}}
\newcommand{\s}{\mathbb{S}}
\newcommand{\rr}{\mathbb{R}}

\newcommand{\N}{\mathbb{N}}
\newcommand{\dx}{\right}
\newcommand{\sx}{\left}

\newcommand{\p}{\partial}

\DeclareMathOperator{\ext}{ext}
\DeclareMathOperator{\prodstar}{\prod \hskip -2.30ex * \hskip 1.3 ex}
\newcommand{\stella}{\prodstar\displaylimits}

\DeclareMathOperator{\ess}{ess}

\title{Quaternionic Hardy Spaces}

\author{Chiara de Fabritiis \footnote{The authors acknowledge the support of G.N.S.A.G.A. of INdAM and MIUR (PRIN Research Project ``Variet\`a reali e complesse: geometria, topologia e analisi armonica''). The third author also acknowledges the support of  MIUR (FIRB Research Project ``Geometria differenziale e teoria geometrica delle funzioni'').}\\
\normalsize Dipartimento di Ing. Industriale e Scienze Matematiche, Universit\`a Politecnica delle Marche\\
\normalsize Via Brecce Bianche 12, 60131 Ancona, Italy,  fabritiis@dipmat.univpm.it\\
\and Graziano Gentili $^*$\\
\normalsize Dipartimento di Matematica e Informatica ``U. Dini'', Universit\`a di Firenze \\
\normalsize Viale Morgagni 67/A, 50134 Firenze, Italy,  gentili@math.unifi.it \\
\and Giulia Sarfatti $^*$\\
\normalsize Dipartimento di Matematica, Universit\`a di Bologna \\
\normalsize Piazza di Porta San Donato 5,  40126 Bologna, Italy,  giulia.sarfatti@unibo.it}

\date{}
\begin{document}

\maketitle

\begin{abstract} The theory of slice regular functions of a quaternionic variable 
extends the notion of holomorphic function to the quaternionic setting. This  fast growing theory is already rich of many results and has significant applications. 
In this setting, the present paper is devoted to introduce and study the quaternionic counterparts of Hardy spaces of holomorphic functions of one complex variable. The basic properties of the theory of  quaternionic Hardy spaces are investigated, and in particular a Poisson-type representation formula, the notions of outer function, singular function and inner function are given. A quaternionic (partial) counterpart of the classical $H^p$-factorization theorem is proved. This last  result assumes a particularly interesting formulation for a large subclass of slice regular functions, where it  is obtained in terms of an outer function, a singular function and a quaternionic Blaschke product.
\end{abstract}




\section{Introduction}

The theory of slice regular functions of a quaternionic variable (often simply called regular functions) was introduced in \cite{G.S.}, \cite{GSAdvances}, and represents a natural quaternionic counterpart of the theory of complex holomorphic functions. This recent theory has been growing very fast: a detailed  presentation appears in the monograph \cite{libroGSS}, while an extension  to the case of real alternative algebras is discussed in \cite{ghiloniperotti1}. 
The theory of regular functions is presently expanding in many directions, and it has in particular been applied to the study of a non-commutative functional calculus, (see for example the monograph \cite{libro2} and the references therein) and to contribute to the problem of the construction and classification of orthogonal complex structures in open subsets of the space of quaternions (see \cite{OCS}).

Let $\HH$ denote the skew field of quaternions. Set  $\s=\{q\in \HH : q^2=-1\}$ to be the $2-$sphere of purely imaginary units in $\HH$, and for $I\in \s$ let $L_I$ be the complex plane $\rr+\rr I$, so that 
$$
\HH=\bigcup_{I\in \s}L_I.
$$
The natural domains of definition for slice regular functions are called  \emph{slice symmetric domains} (see \cite{libroGSS}). In this paper we will consider slice regular functions defined on the open unit ball $\B=\{q\in \HH : |q|<1 \}$, which is a particular example of slice symmetric domain.

\begin{defn}
A function $f: \B \rightarrow \HH$ is said to be {\em (slice) regular} if, for all $I\in \mathbb{S}$,  its restriction $f_I$ to $\B_I=\B\cap L_I$ is \emph{holomorphic}, i.e., it has continuous partial derivatives and satisfies
$$\overline{\partial}_I f(x+yI):=\frac{1}{2}\Big(\frac{\partial}{\partial x}  +I\frac{\partial}{\partial y}\Big)f_I(x+yI)=0$$
for all $x+yI \in \B_I$.
\end{defn}
\noindent As shown in \cite{GSAdvances}, if the domain of definition is the open unit ball $\B$ of $\HH$, the class of regular functions coincides with the class of convergent power series of type $\sum_{n\ge 0} q^na_n$, with all $a_n\in \HH$.

One of the most fertile chapters of the theory of complex holomorphic functions consists of the theory of Hardy spaces. This theory contains results of great significance that led to important general achievements and subtle applications.
In this paper we define the quaternionic counterpart of complex Hardy spaces, and investigate their basic and fundamental properties. Our definition is the following.
\begin{defn}
Let $f: \B \to \HH$ be a regular function  and let $0<p<+\infty$. Set  
\begin{equation*}
||f||_p=\sup_{I\in \s}\lim_{r\to 1^-}\left(\frac{1}{2\pi} \int_0^{2\pi}|f(re^{I\theta})|^p d\theta\right)^{\frac{1}{p}},
\end{equation*}
and set 
\[||f||_{\infty}=\sup_{q\in \B}|f(q)|.\]
Then, for any $0<p\le +\infty$, we define the quaternionic Hardy space $H^p(\B)$ as
\[H^p(\B)=\{f:\B\to \HH \, \, | \, \,   \text{$f$ is regular and $||f||_p<+\infty$}\}.\]
\end{defn}
\noindent In Section \ref{spaces} we study the main properties and features of the quaternionic $H^p$-norms, we motivate the chosen definition and establish the initial properties of the quaternionic $H^p$ spaces. In Section \ref{boundary} we investigate the
boundary behavior of functions $f$ in $H^p(\B)$, obtaining that  for almost every $\theta\in \rr$, the limit
\[\lim_{r\to 1^-}f(re^{I\theta})=\tilde f_I(e^{I\theta})\]
exists for all $I\in \s$ and in this case it
belongs to $L^p(\p \B_I)$. We then investigate the properties of the boundary values of the $*$-product of two functions, each belonging to some $H^p(\B)$ space.
In section \ref{factorization} possible analogues of {\em outer} and {\em inner} functions and {\em singular} factors on $\B$ are given, whose definitions (when compared with those used in the complex case) clearly resent of the peculiarities of the non commutative quaternionic setting.  
Our  results here include factorization properties of $H^p$ functions. We are able to identify the {\em Blaschke factor} of a function $f$ in $H^p(\B)$, built from the zero set of $f$; then we can exhibit a complete factorization result, in terms of an outer, a singular  and a Blaschke factor, for a subclass of regular functions, namely for the one-slice-preserving functions. 

The cases of $H^2(\B)$ (defined already in \cite{milanesi2}) and $H^{\infty}(\B)$ are special cases, as in the complex setting. The Hardy spaces $H^2(\B)$ have been also treated in \cite{3,4,1,2}, and the space $H^{\infty}(\B)$ briefly appears in  \cite{1}. 
The case of Hardy spaces $H^2$ defined on the half-space is considered in \cite{4, 2}, where the aspect of the reproducing kernels is also investigated. 
Integral representations of slice hyperholomorphic functions have been recently studied in \cite{Gonzales}.

The "slicewise" approach has been also  used in \cite{6,7,8,9} to treat quaternionic Bergman spaces, and in \cite{10,5} to introduce and discuss the analogs of Besov, Bloch, Dirichlet and Fock spaces.
Hardy and Dirichlet spaces in the more general setting of Clifford valued monogenic functions are studied for example in \cite{bernstein}. 

\section{Preliminaries}

We recall in this section some preliminary definitions and results that are essential for the comprehension of the rest of the paper.  All these preliminaries are stated, for the sake of simplicity,  in the particular case of the open unit ball even if they hold in the more general setting of slice symmetric domains. 
The related proofs can be found in the literature, and in particular in \cite{libroGSS}.
Let us begin with a basic result that establishes the connection between the class of regular functions and the class of complex holomorphic 
functions of one complex variable. 
\begin{lem}[Splitting Lemma]\label{split}
If $f$ is a regular function on $\B$, then for every $I \in \mathbb{S}$ and for every $J \in \mathbb{S}$, $J$ orthogonal to $I$, 
there exist two holomorphic functions $F,G:\B_I \rightarrow L_I$, such that for every $z=x+yI \in \B_I$, we have 
\[f_I(z)=F(z)+G(z)J.\]
\end{lem}
%

\noindent We can recover the values of a regular function defined on $\B$ from its values on a single 
slice $L_I$ (see \cite{kernel, ext}).

\begin{teo}[Representation Formula]\label{RF}
Let $f$ be a regular function on  $\B$ and let $I\in \mathbb{S}$. Then, for all $x+yJ \in \B$ with $J\in  \mathbb{S}$,
 the following 
equality holds
\begin{equation*}
f(x+yJ)=\frac{1}{2}\left[ f(x+yI)+f(x-yI)\right]+\frac{JI}{2}\left[f(x-yI)-f(x+yI) \right].
\end{equation*}
In particular the function $f$ is {\em affine} on each sphere of the form $x+y\s$ contained in $\B$, namely there exist $b, c \in \HH$ such that $f(x+yJ)=b+Jc$ \;  for all $J\in \s$ . 
\end{teo}
\noindent Thanks to the Representation Formula it is possible to prove that we can estimate the maximum modulus of a function with its maximum 
modulus on each slice, (see \cite{weierstrass}).
\noindent In the special case of a regular function $f$ that maps a slice $L_I$ to itself, 
we obtain that the maximum (and minimum) modulus 
of $f$ is actually attained on the preserved slice.
\begin{pro}\label{maxminslice}
Let $f$ be a regular function on $\B$ such that $f(\B_I)\subset L_I$ for some $I\in \s$. 
Then, for any $x+y\s \subset \B$,
\[\max_{J\in \s}|f(x+yJ)|=\max\{|f(x+yI)|, |f(x-yI)|\}\]
and
\[\min_{J\in \s}|f(x+yJ)|=\min\{|f(x+yI)|, |f(x-yI)|\}.\]
\end{pro}

\noindent Holomorphic functions defined on a disc $\B_I$,  in the complex plane $L_I$, extend uniquely to $\B$. 

\begin{lem}[Extension Lemma]\label{extensionlemma}
Let  $I \in \s$. If $f_I: \B_I \rightarrow \HH$ is holomorphic, then setting
\begin{equation*}
f(x+yJ)=\frac{1}{2}[f_I(x+yI)+f_I(x-yI)] +J\frac{I}{2}[f_I(x-yI)-f_I(x+yI)]
\end{equation*}
extends $f_I$ to a regular function $f:\B \rightarrow \HH$. Moreover $f$ is the unique extension and it is denoted by $\ext(f_I)$.
\end{lem}

\noindent The pointwise product of two regular functions is not, in general, regular. To guarantee the regularity we have to use a different  multiplication operation, the $*$-product, which extends the one classically defined for (polynomials and) power series with coefficients in a non-commutative field.


\begin{defn}\label{R-starprodotto}\index{regular product} \index{$*$-product}
Let $f,g$ be regular functions on $\B$. Choose $I,J \in \s$ with $I \perp J$ and let $F,G,H,K$ be holomorphic functions from $\B_I$ to $L_I$ such that $f_I = F+GJ, g_I = H+KJ$. Consider the holomorphic function defined on $\B_I$ by
\begin{equation*}
f_I*g_I(z) = \left[F(z)H(z)-G(z)\overline{K(\bar z)}\right] + \left[F(z)K(z)+G(z)\overline{H(\bar z)} \right]J.
\end{equation*}
Its regular extension $\ext(f_I*g_I)$ is called the \emph{regular product} (or \emph{$*$-product}) of $f$ and $g$ and it is denoted by $f*g$.
\end{defn}

\noindent Notice that the $*$-product is associative but generally is not commutative. Its connection with the usual pointwise product is stated by the following result.

\begin{pro}\label{trasf}
Let $f$ and $g$ be regular functions on $\B$. Then
\begin{equation}\label{prodstar}
f*g(q)= \left\{ \begin{array}{ll}
 f(q)g(f(q)^{-1}qf(q)) & \text{if} \quad f(q)\neq 0\\
0 & \text{if} \quad f(q)=0
\end{array}
\right.
\end{equation}
In particular $f*g(q) = 0$ if and only if $f(q) = 0$ or $f(q) \neq 0$ and $g(f(q)^{-1} q f(q))=0$.
\end{pro}

\noindent The regular product coincides with the pointwise product for the  special class of regular functions defined as follows.

\begin{defn}\label{R-slicepreserving}\index{slice preserving regular function}
A regular function $f: \B \to \HH$ such that $f(\B_I)\subseteq L_I$ for all $I \in \s$  is called a \emph{slice preserving} regular function.
\end {defn}

\begin{lem}\label{realproduct}
Let $f,g$ be regular functions on $\B$. If $f$ is slice preserving, then $fg$ is a regular function on $\B$ and $f*g= fg=g*f$.
\end{lem}

%


\noindent The following operations are naturally defined in order to study the zero set of regular functions.

\begin{defn}\label{f^c}
Let $f$ be a regular function on  $\B$ and suppose that for $z\in \B_I$, the splitting of $f$ 
with respect to $J$ ($J\in \s$ orthogonal to $I$) is $f_I(z)=F(z)+G(z)J$. Consider the holomorphic function
$f_I^c(z)=\overline{F(\overline{z})}-G(z)J.$
The {\em regular conjugate} of $f$ is the function defined  by  
$f^c(q)=\ext(f_I^c)(q).$
The {\em symmetrization} of $f$ is the function defined by
$f^s(q)=f*f^c(q)=f^c*f(q).$
\noindent Both $f^c$ and $f^s$ are regular functions on $\B$.
\end{defn}

\noindent We now recall a result (proven in \cite{sarfatti}) that relates the norm of a regular function $f:\B \to \HH$ and that of its regular conjugate.
\begin{pro}\label{norma}
Let $f$ be a regular function on $\B$. For any sphere of the form $x+y\s$ contained in $\B$, the following equalities hold true 
\[\max_{I \in\s}|f(x+yI)|=\max_{I \in\s}|f^c(x+yI)| \quad \ \text{and} \ \quad \min_{I \in \s}|f(x+yI)|=\min_{I \in \s}|f^c(x+yI)|.
\]
\end{pro}

\noindent Before we give the complete characterization of the zero set of a regular function, let us recall another result in this setting, 
consequence of the Representation Formula \ref{RF}.
\begin{pro}\label{zerislice}
Let $f$ be a regular function on $\B$ such that $f(\B_I)\subset L_I$ for some $I\in\s$. 
If $f(x+yJ)=0$ for some $J\in\s\setminus \{\pm I\}$, then $f(x+yK)=0$ for any $K\in\s$.
\end{pro}

\begin{teo}[Zero set structure]\label{zeri}
Let $f$ be a regular function on  $\B$. If $f$ does not vanish identically, 
then its zero set consists of the union of isolated points and isolated $2$-spheres of the form $x +y \mathbb{S}$ with 
$x,y \in \mathbb{R}$, $y \neq 0$.
\end{teo} 

\noindent Spheres of zeros of real dimension $2$ are a peculiarity  of regular functions.

\begin{defn}\label{generatori}
Let $f$ be a regular function on  $\B$. A $2$-dimensional sphere $x+y\s \subset \B$ of zeros of f is called a 
{\em spherical zero} of $f$. 
Any point $x+yI$ of such a sphere is called a {\em generator} of the spherical zero $x+y\s$.
Any zero of $f$ that is not a generator of a spherical zero is called an
{\em isolated zero} (or a {\em non spherical zero} or simply a {\em zero}) of $f$.
\end{defn}
%
\noindent An appropriate notion of multiplicity, introduced in \cite{milan} for the case of regular quaternionic polynomials, can be given for zeros of 
regular functions.

\begin{defn}\label{multiplicity}
Let $f$ be a regular function on  $\B$ and let $x+y\s\subset \B$ with $y\neq0$. 
Let $m,n\in\N$ and $p_1,\dots,p_n\in x+y\s$ (with $p_i\neq \overline{p}_{i+1}$ for all $i\in\{1,\dots,n-1\}$) be such that 
\[f(q)=\left((q-x)^2+y^2\right)^m(q-p_1)*(q-p_2)*\dots*(q-p_n)*g(q)\]
for some regular function $g:\B\to \HH$ which does not have zeros in $x+y\s$. Then
$2m$ is called the {\em spherical multiplicity} of $x+y\s$ and $n$ is called the {\em isolated multiplicity} of $p_1$. 
On the other hand, if $x\in\rr$, then we call {\em isolated
multiplicity} of $f$ at $x$ the number $k\in\N$ such that
 \[f(q)=(q-x)^kh(q)\]
for some regular function $h:\B \to \HH$ which does not vanish at $x$.
\end{defn}

%

\noindent The regular conjugate, and the symmetrization, of a regular function $f$ on  $\B$ allow the definition of its  \emph{regular reciprocal} $f^{-*}$ defined on $\B\setminus Z_{f^s}$ as
\begin{equation*}
f^{-*}=(f^{s})^{-1}f^c,
\end{equation*}
where $Z_{f^s}$ denotes the zero set of the symmetrization $f^s$. This naturally leads to the definition of \emph{regular quotient} of regular functions, for more details see \cite{libroGSS}.

\begin{pro}\label{Caterina} 
Let $f$ and $g$ be regular functions on $\B$. 
If \  $T_f : \B \setminus Z_{f^s} \rightarrow  \B \setminus Z_{f^s}$ is defined as
\[T_{f}(q)=f^c(q)^{-1}qf^c(q),\] then 
\[f^{-*}*g(q)=f(T_{f}(q))^{-1}g(T_{f}(q)) \quad \text{for every} \quad q \in \B \setminus Z_{f^s}.\] 
Furthermore, $T_f$ and $T_{f^c}$ are mutual inverses so that $T_f$ is a diffeomorphism.
\end{pro}


\noindent Regular Moebius transformations are interesting examples of regular quotients (see \cite{regular}). 

\begin{defn}\label{moebius}
A {\em regular Moebius transformation} is a regular function $M:\B\to\B$ of the form
\[M(q)=(1-q\overline{a})^{-*}*(q-a)u\]
where $a\in\B$ and $u\in\p \B$.
\end{defn}

\noindent In particular, these transformations are regular bijections of the open unit ball onto itself and, by direct computation, 
it is possible to prove that they map the boundary of the unit ball to itself.

\begin{pro}\label{moeb2}
A function $M$ is a regular bijection from $\B$ to itself if and only if $M$ is a regular Moebius transformation.
\end{pro}

\section{The quaternionic spaces $H^p(\B)$}\label{spaces}

In this section we give a definition of quaternionic  Hardy spaces. The first natural step is to define an appropriate Hardy-type norm.
Let $p\in (0,+\infty)$, $r\in[0,1)$, and $I\in \s$. For any regular function $f$ in the unit ball $\B$, let $(f_I)_r$ be the function defined on the unit circle $\p \B_I$ by
\[(f_I)_r(e^{I\theta})=f(re^{I\theta})\]
and let $M_p(f_I,r)$ be the integral mean
\begin{equation}\label{mediaint}
M_p(f_I,r)=\left(\frac{1}{2\pi} \int_{-\pi}^{\pi}|(f_I)_r(e^{I\theta})|^p d\theta\right)^{\frac{1}{p}}=\left(\frac{1}{2\pi} \int_{-\pi}^{\pi}|f(re^{I\theta})|^p d\theta\right)^{\frac{1}{p}}.
\end{equation}
If $p=+\infty$, set $M_{\infty}(f,r)$ to be defined as
\[M_{\infty}(f,r)=\sup_{|q|<r}|f(q)|\]
for $0<r<1$ (and $|f(0)|$ for $r=0$).
\begin{pro}\label{incr}
Let $f:\B\to\HH$ be a regular function. Then for any $p\in (0,+\infty)$ and for any imaginary unit $I\in\s$, the function $r \mapsto M_p(f_I,r)$ 
is increasing. For $p=+\infty$, the function
$r \mapsto M_{\infty}(f,r)$ is increasing as well.
\end{pro}
\begin{proof}
The last part of the statement follows directly from the Maximum Modulus Principle for regular functions, \cite{libroGSS}.
To prove the first part, fix $p\in(0,+\infty)$ and $I\in \s$. We will prove that the function 
\[g:\B_I \to \rr, \quad g:z\mapsto |f_I(z)|^p,\]
is a subharmonic function, and then the statement will follow from classical results, see for instance Theorem 17.5 in \cite{Rudin}. 
In order to obtain the subharmonicity of $g(z)$, we will first show that 
\begin{equation}\label{hsub}
h:\B_I \to \rr, \quad h:z\mapsto \log(|f_I(z)|)
\end{equation}
is a subharmonic function. Then, since  $g(z)=e^{ph(z)}$ is the composition of an increasing convex function with a subharmonic function, we will conclude that also $g$ is subharmonic (see e.g., Theorem 17.2 in \cite{Rudin}).\\ 
Suppose $f \not\equiv 0$ (otherwise the statement is trivially true). 
Then $h(z)$ is an upper semicontinuous function. Moreover in the set $\{z\in \B_I \, | \, f_I(z) \neq 0\}$ the function $h$ is subharmonic. 
In fact we have that
\begin{equation*}
\begin{aligned}
\Delta h(z)= 4 \frac{\partial}{\partial \overline{z}} \frac{\partial}{\partial z}\log(|f_I(z)|)= 
2 \frac{\partial}{\partial \overline{z}} \frac{\partial}{\partial z}\log(|F(z)|^2+|G(z)|^2)
\end{aligned}
\end{equation*}
where $F,G$ are the splitting (holomorphic) functions of $f_I$ with respect to $J\in \s$, $J$ orthogonal to $I$. 
 Hence (omitting the variable $z$)
\begin{equation}\label{sub}
\begin{aligned} 
\Delta h&
=2\frac{\left(|F'|^2+|G'|^2\right)\left(|F|^2+|G|^2\right)-\left(F'\overline{F}+G'\overline{G}\right)\left(F\overline{F'}+G\overline{G'}\right)}{\left(|F|^2+|G|^2\right)^2}.
\end{aligned}
\end{equation}
Schwarz inequality 
yields that $h(z)$ is subharmonic  where $f_I$ is non vanishing. 
It remains to show that $h$ is still subharmonic in a neighborhood of each zero of $f_I$. Recall that the zero set of a regular function intersected with a slice $L_I$, is a discrete subset of $L_I$ (see Theorem \ref{zeri}). Then for any zero $z_0$ of $f_I$, there exists a neighborhood $U_I \subset \B_I$ where $z_0$ is the only point where $f_I$ vanishes. Hence we have that for all $r$ such that $z_0+B_I(0,r) \subset U_I$ the submean property is trivially satisfied
\[-\infty=h(z_0) \leq \frac{1}{2\pi}\int_{\partial B_I(0, r)}h(z_0+re^{I\theta})d\theta.\]
Since this condition implies the subharmonicity of $h$ near $z_0$, we can conclude the proof. 
\end{proof}

\begin{oss}
The previous result is the analogue of the first statement of the Hardy convexity Theorem in the complex setting, see Theorem 1.5 in \cite{duren}.
\end{oss}

\begin{oss}
We point out that, despite what happens in the complex case, where the function $z\mapsto \log(|g(z)|)$ is 
actually harmonic for any non vanishing holomorphic function $g$, the function $h$ defined in equation  \eqref{hsub} is only subharmonic in general.
 In fact $\Delta h=0$ if and only if 
 the vector 
$(F(z),G(z))\in (L_I)^2$, identified by the restriction of the regular function $f_I$, is parallel to the vector $(F'(z),G'(z))$ 
identified by the derivative of $f_I$. This happens for example if 
\begin{equation*}
\left\{
\begin{array}{l}
F'(z)=kF(z)\\
G'(z)=kG(z)
\end{array}
\right.
\end{equation*}
for some $k\in L_I$.
In this case $F$ and $G$ are exponential (or constant) functions, 
\begin{equation*}
\left\{
\begin{array}{l}
F(z)=F(0)e^{kz}\\
G(z)=G(0)e^{kz}.
\end{array}
\right.
\end{equation*}

\end{oss}

Thanks to Proposition \ref{incr}, we can give the following Definition.
\begin{defn}
Let $f:\B\to\HH$ be a regular function. If $p\in(0,+\infty)$, for any $I\in \s$, we set 
\[||f_I||_p=\lim_{r\to 1^-}M_p(f_I,r)=\lim_{r\to 1^-}\left(\frac{1}{2\pi} \int_0^{2\pi}|f(re^{I\theta})|^p d\theta\right)^{\frac{1}{p}}
\]
and
\begin{equation*}
||f||_p=\sup_{I\in \s}||f_I||_p=\sup_{I\in \s}\lim_{r\to 1^-}\left(\frac{1}{2\pi} \int_0^{2\pi}|f(re^{I\theta})|^p d\theta\right)^{\frac{1}{p}}.
\end{equation*}
If $p=+\infty$, we set 
\[||f||_{\infty}=\lim_{r\to 1^-}M_{\infty}(f,r)=\lim_{r\to 1^-}\sup_{|q|<r}|f(q)|=\sup_{q\in \B}|f(q)|.\]
\end{defn}
\begin{oss}
Notice that $||f||_{\infty}$ is the uniform norm of $f$ on $\B$.
Moreover, if we set for any $I\in \s$ 
\[||f_I||_{\infty}= \sup_{z\in \B_I}|f(z)|,\]
then we have
\[||f||_{\infty}=\sup_{I\in\s}||f_I||_{\infty}.\]
The set where the uniform norm is taken will be $\B_I$ when considering the restriction $f_I$ (or its splitting components), 
and $\B$ when considering the function $f$.
\end{oss}

\begin{defn}\label{Hp}
Let $p\in (0,+\infty]$. We define the quaternionic Hardy space $H^p(\B)$ as
\[ H^p(\B)= \left\{ f\colon \B\to \HH \, \ | \, \ f \, \text{is regular and }\, \ ||f||_p< +\infty \right\}.\] 
\end{defn}

\begin{oss}\label{inclusion}
For any $p\in(0,+\infty]$, the space $H^p(\B)$ is a real vector space. Furthermore  if (and only if) $p\ge 1$ the function $||\cdot ||_p$ satisfies the triangle inequality, and hence it is a norm on $H^p(\B)$. 
Moreover, the same relations of inclusions that hold for complex $H^p$ spaces, hold in the quaternionic setting. 
In fact, for any $p,q$ such that $0<p<q\le +\infty$, thanks to the classical Jensen inequality we have that
\[H^q(\B)\subset H^p(\B).\]
The inclusion is continuous for $1\le p<q\le +\infty$.
\end{oss}
In analogy with the complex case, the space $H^2(\B)$ is special. 
Indeed the $2$-norm turns out to be induced by an inner product (see \cite{milanesi}). 
\begin{pro}\label{norma2}
Let $f\in H^2(\B)$ and let $f(q)=\sum_{n\ge 0}q^na_n$ be its power series expansion. Then the $2$-norm of $f$,
\[||f||_2=\sup_{I\in \s}\lim_{r\to1^-}\left(\frac{1}{2\pi}\int_{-\pi}^{\pi}|f(re^{I\theta})|^2d\theta \right)^{\frac{1}{2}},\]
coincides with 
\[\Big(\sum_{n\ge 0}|a_n|^2\Big)^{\frac{1}{2}}.\]
\end{pro}
\noindent By polarization we obtain
 that, for any $I\in\s$
\[\lim_{r\to 1^-}\frac{1}{2\pi}\int_{-\pi}^{\pi}\overline{g(re^{I\theta})}f(re^{I\theta})d\theta=\sum_{n\ge 0}\overline{b_n}a_n = \langle f, g \rangle \]
that recalls the classic Hermitian product of the space $H^2(\D)$.

Looking at the definition of complex Hardy spaces, one could wonder why, instead of the integral mean $M_p(f_I,r)$ defined in \eqref{mediaint}, taken on a circle, we did not choose the classical integral mean taken on a $3$-dimensional sphere. In fact consider 
\begin{equation}\label{normaS}
N_p(f,r)=\left(\frac{1}{2\pi^2r^3}\int_{r\mathbb{\s}^3} |f(q)|^p d\sigma_3(r\s^3)\right)^{\frac{1}{p}},
\end{equation}
where $\sigma_3(r\s^3)$ is the usual hypersurface measure of the $3$-dimensional sphere $r\mathbb{S}^3$ and $2\pi^2r^3$ is the $3$-dimensional volume of the $3$-dimensional sphere $r\mathbb{S}^3$. If we then set
\[N_p(f)=\sup_{0<r<1}N_p(f,r),\]
it turns out that the class of regular functions $f$ such that $N_p(f)$ is finite does not coincide with the space $H^p(\B)$ defined earlier. Indeed,
using an appropriate change of variable in the integral appearing in \eqref{normaS}, it is possible to prove 
that $N_p(f) \le ||f||_p$ for any regular function $f$ and any $p\in(0+\infty)$, thus showing that if $f\in H^p(f)$ then also $N_p(f)$ 
is finite.
\noindent To show that the two classes of functions do not coincide, it is possible to exhibit explicit examples. For instance, the function
\[f(q)=(1-q)^{-*}\] 
is such that $N_p(f)<+\infty$ but it does not belong to $H^p(\B)$ for all $2\leq p\leq +\infty$.
\noindent This example suggests us that a definition of a quaternionic $H^p$ space that relies upon the $N_p$ norm is not convenient. 
In fact, the function $(1-q)^{-1}$ has real coefficients, and therefore, its restriction to $\B_I=\B\cap L_I$  is the holomorphic function
$f_I(z)=(1-z)^{-1}.$
It is well known that $f_I(z)$ does not belong to $H^p(\B_I)$ for all $2\leq p\leq +\infty$, hence it would be weird if its regular extension were in the quaternionic $H^p$ space. Recently, a normalized integral mean on $3$-spheres has been introduced to define a norm on $H^2(\B)$, \cite{artAS}.

Our choice \eqref{mediaint} of the $p$-integral mean allows the slicewise approach to the theory of $H^p$ spaces. 
Let us begin by studying how $||f||_p$ of a regular function $f$ is related with the value of $||f_I ||_p$ on the slice $L_I$. 
\begin{pro}\label{f_Ihp}
A regular function $f$ is in $H^{p}(\B)$ for some $p \in (0,+\infty]$, if and only if there exists $I \in \s$ such that $||f_I||_p$ is bounded.
\end{pro}
\begin{proof}
Let $p\in (0,+\infty)$ and suppose first that $||f||_p<+\infty.$  Then trivially, for any $I\in \s$, 
\begin{equation*}
\begin{aligned}
||f_I||_p\le \sup_{I\in\s}||f_I||_p=||f||_p<+\infty.
\end{aligned}
\end{equation*}
To see the other implication, let  $J \in \s$ be such that $||f_J||_p$ is finite, and use the Representation Formula \ref{RF} to write
\begin{equation}\label{f_IHp}
\begin{aligned}
||f||^p_p&= 
 \sup_{I\in \s}\lim_{r\to 1^{-}}\frac{1}{2\pi} \int_{-\pi}^{\pi}\left |\frac{1}{2}\left(f(re^{J\theta})+f(re^{-J\theta})\right)+\frac{IJ}{2}\left(f(re^{-J\theta})-f(re^{J\theta})\right)\right |^p d\theta\\
&\leq \sup_{I\in \s}\lim_{r\to 1^{-}}\frac{1}{2\pi} \int_{-\pi}^{\pi}\left(\frac{1}{2}\left(\left|f(re^{J\theta})\right|+\left|f(re^{-J\theta})\right|+\left|f(re^{-J\theta})\right|+\left|f(re^{J\theta})\right|\right)\right)^p d\theta\\
&=\lim_{r\to 1^{-}}\frac{1}{2\pi} \int_{-\pi}^{\pi}\left(\left|f(re^{J\theta})\right|+\left|f(re^{-J\theta})\right|\right)^p d\theta
\end{aligned}
\end{equation}
where the last integral does not depend on $I\in \s$. 
If $p\geq 1$, taking into account the convexity of the map $x\mapsto x^p$ on the positive real axis, we get that 
\begin{equation*}
\begin{aligned}
||f||^p_p& 
\leq \lim_{r\to 1^{-}}\frac{2^{p-1}}{2\pi} \int_{-\pi}^{\pi}\left(\left|f(re^{J\theta})\right|^p+\left|f(re^{-J\theta})\right|^p\right)d\theta
=2^{p}||f_J||^p_p<+ \infty .
\end{aligned}
\end{equation*}
On the other hand, if $0<p<1$, taking into account the subadditivity on the positive real axis of the map $x\mapsto x^p$ (it is concave and maps $0$ to $0$), we have
\[||f||^p_p
\le\lim_{r\to 1^{-}}\frac{1}{2\pi} \int_{-\pi}^{\pi}\left(\left|f(re^{J\theta})\right|^p+\left|f(re^{-J\theta})\right|^p \dx)d\theta
= \lim_{r\to 1^{-}}\frac{1}{\pi} \int_{-\pi}^{\pi}\sx|f(re^{J\theta})\dx|^pd\theta=2||f_J||_p^p <+ \infty .
\]
Let $p=+\infty$. On one side, if $f$ is bounded, then for any $I\in\s$
\[||f_I||_{\infty}=\sup_{z\in \B_I}|f_I(z)|\le\sup_{q\in \B}|f(q)|=||f||_{\infty}<+\infty.\]
On the other side, using the Representation Formula \ref{RF}, if $J\in \s$ is such that $||f_J||_{\infty}<+\infty$ it is easy to see that
\[||f||_{\infty}\le 2||f_J||_{\infty}<+\infty.\]

\end{proof}

\begin{oss}\label{stima}
In particular, we get that if $||f_J||_p<+\infty$ for some $J\in \s$, we have the following inequalities
\begin{equation*}
||f_J||_p\leq ||f||_p \leq 2^{\frac{1}{p}}||f_J||_p,   \text{ if $p\in(0,1)$,} \qquad ||f_J||_p\leq ||f||_p \leq 2||f_J||_p,   \text{ if $p\in [1,+\infty]$.}
\end{equation*}
\end{oss}
%
The key fact that allows us to apply classical results to the splitting components of a function in $H^p(\B)$ is the following.
\begin{pro}\label{FGHp}
Let $f\in H^p(\B)$ for some $p\in(0,+\infty]$. Then for any $I\in\s$, if the splitting of $f$ on $L_I$ with respect to $J\in \s$, $J \perp I$,  is $f_I(z)=F(z)+G(z)J$, then the holomorphic functions $F$ and $G$ are both in $H^p(\B_I)$. 
\end{pro}
\begin{proof}
Let $I\in\s$ and consider, for any $z\in\B_I$, the splitting $f_I(z)=F(z)+G(z)J$. Then, for any $z\in \B_I$,
\begin{equation}\label{Finfty}
|f_I(z)|=\sqrt{|F(z)|^2+|G(z)|^2}\ge\max\{|F(z)| ,|G(z)|\}.
\end{equation}
Hence, for $p\in(0,+\infty)$,
\begin{equation*}
\begin{aligned}
+\infty &>||f||^p_p
=\lim_{r \to 1^-} \frac{1}{2\pi} \int_{-\pi}^{\pi}\sx|f_I(re^{I\theta})\dx|^p d\theta
\ge \lim_{r\to 1^-} \frac{1}{2\pi} \int_{-\pi}^{\pi}\sx|F(re^{I\theta})\dx|^p d\theta= ||F||^p_p
\end{aligned}
\end{equation*}
and analogously
$+\infty>||f||^p_p \ge ||f_I||^p_p\ge ||G||_p^p.$
For $p=+\infty$, formula \eqref{Finfty} directly implies that
\[ ||F||_{\infty}\le ||f_I||_{\infty}\le||f||_{\infty}<+\infty
\quad \text{ and } \quad 
||G||_{\infty}\le ||f_I||_{\infty}\le||f||_{\infty}<+\infty.\]

\end{proof}

Notice that Remark \ref{stima} and Proposition \ref{FGHp} imply that, for all $p\in [1,+\infty]$, the space  $H^p(\B)$ is a Banach space.

The natural guess that if a function is in $H^p(\B)$ then its regular conjugate is in $H^p(\B)$ as well, is in fact true.
\begin{pro}\label{fcHp}
Let $p\in (0,+\infty]$ and let $f\in H^p(\B)$. Then also the regular conjugate $f^c$ belongs to $H^p(\B)$.
\end{pro}
\begin{proof}
If $p=+\infty$ the proof follows directly by Proposition \ref{norma}.
Consider then $p\in(0,+\infty)$. For any $I\in \s$, if $f$ splits on $\B_I$ as
$f(z)=F(z)+G(z)J,$
then, recalling Definition \ref{f^c}, we get that on the same slice $f^c$ can be written as
\[f_I^c(z)=\overline{F(\overline{z})}-G(z)J.\]
Therefore, for $0\leq r<1$,
\begin{equation*}
\begin{aligned}
M_p(f_I^c,r)^p
=\frac{1}{2\pi}\int_{-\pi}^{\pi}\left|\overline{F(re^{-I\theta})}-G(re^{I\theta})J\right|^pd\theta
=\frac{1}{2\pi}\int_{-\pi}^{\pi}\left(\left|\overline{F(re^{-I\theta})}\right|^2+\left|G(re^{I\theta})\right|^2\right)^{\frac{p}{2}}d\theta.
\end{aligned}
\end{equation*}
If $0<p<2$, thanks to the subadditivity on the positive real axis of the map $x\mapsto x^{p/2}$, we get
\begin{equation*}
\begin{aligned}
M_p(f^c_I,r)^p&
\le\frac{1}{2\pi}\int_{-\pi}^{\pi}\left(\left|\overline{F(re^{-I\theta})}\right|^p+\left|G(re^{I\theta})\right|^p\right)d\theta
= M_p(F,r)^p+M_p(G,r)^p.
\end{aligned}
\end{equation*}
Proposition \ref{FGHp} yields that both $F$ and $G$ belong to $H^p(\B_I)$, hence 
\begin{equation*}
\begin{aligned}
||f_I^c||^p_p&=\lim_{r\to1^-}M_p(f_I^c,r)^p\le\lim_{r\to1^-}\sx(M_p(F,r)^p+M_p(G,r)^p\dx)=||F||^p_p+||G||_p^p<+\infty
\end{aligned}
\end{equation*}
and therefore we obtain that $f^c(q)\in H^p(\B).$\\
If $2\le p <+\infty$, thanks to the convexity of the map $x\mapsto x^{p/2}$, we can bound $M_p(f_I^c,r)$ as follows
\[
M_p(f_I^c,r)^p
\le \frac{2^{\frac{p}{2}-1}}{2\pi}\int_{-\pi}^{\pi}\left(\left|\overline{F(re^{-I\theta})}\right|^p+\left|G(re^{I\theta})\right|^p\right) d\theta =2^{\frac{p}{2}-1}\left(M_p(F,r)^p+M_p(G,r)^p\right).
\]
Hence, as before,
\begin{equation*}
\begin{aligned}
||f_I^c||^p_p
&\le2^{\frac{p}{2}-1}\left(||F||^p_p+||G||_p^p\right)<+\infty,
\end{aligned}
\end{equation*}
which concludes the proof.
\end{proof}
For the symmetrization of a function in $H^p(\B)$ the following result holds true. 
\begin{pro}\label{fsbounded}
For any $p\in(0,+\infty)$, if $f\in H^p(\B)$, then the symmetrization $f^s\in H^{\frac{p}{2}}(\B)$. Moreover if $f\in H^{\infty}(\B)$ then also $f^s$ does. 
\end{pro}
\begin{proof} Let $f\in H^p(\B)$ for some $0<p<+\infty$.
For any $r\in[0,1), I\in \s$ (such that $f(re^{I\theta})\neq 0$) we have
\begin{equation*}
\left|f^s(re^{I\theta})\right|=\left|f*f^c(re^{I\theta})\right|=\left|f(re^{I\theta})f^c(f(re^{I\theta})^{-1}re^{I\theta}f(re^{I\theta}))\right|=\left|f(re^{I\theta})f^c(re^{J\theta})\right|
\end{equation*}
for some $J\in\s$.
Hence, recalling Proposition \ref{norma},
\begin{equation}\label{modulofs}
\left|f^s(re^{I\theta})\right|\le\sup_{I\in\s}\left|f(re^{I\theta})\right|\sup_{J\in\s}\left|f^c(re^{J\theta})\right|=\left(\sup_{I\in\s}\left|f(re^{I\theta})\right|\right)^2=\left|f(re^{K(r,\theta)\theta})\right|^2
\end{equation}
for a suitable $K(r,\theta)\in \s$ which depends on $r,\theta$ but does not depend on $I$ (since $\s$ is compact).
For such a $K(r,\theta)$ inequality \eqref{modulofs} implies
\[M_{\frac p2}(f^s_I, r)^{\frac p2}=
 \frac{1}{2\pi}\int_{-\pi}^{\pi}\left|f^s(re^{I\theta})\right|^{\frac{p}{2}} d\theta\leq
\frac{1}{2\pi}\int_{-\pi}^{\pi}\left|f(re^{K(r,\theta)\theta})\right|^p d\theta .\]
Let now $L$ be any (fixed) imaginary unit independent of $K(r,\theta)$. Thanks to the Representation Formula \ref{RF} we can write  
\[f(re^{K(r,\theta)\theta})=\frac{1}{2}\left(f(re^{L\theta})+f(re^{-L\theta})\right)+\frac{K(r,\theta)L}{2}\left(f(re^{-L\theta})-f(re^{L\theta})\right),\]
and hence we have
\[M_{\frac p2}(f^s_I, r)^{\frac p2}\leq\frac{1}{2\pi}\int_{-\pi}^{\pi}\left(\left|(f(re^{L\theta})\right|+\left|f(re^{-L\theta})\right|\right)^p d\theta.\]
Then, if $p\in(0,1)$, we get that
\[M_{\frac p2}(f^s_I, r)^{\frac p2}\leq 2 ||f_L||^p_p\leq 2||f||_p^p<+\infty,\]
and if $p\in[1,+\infty)$, we get that 
\[M_{\frac p2}(f^s_I, r)^{\frac p2}\leq 2^{p} ||f_L||^p_p\leq 2^p||f||_p^p<+\infty.\]
In both cases, taking the limit for $r\to 1^-$, and the supremum in $I\in \s$,  we get that $f^s\in H^{\frac{p}{2}}(\B)$.

If instead $f\in H^{\infty}(\B)$, we observe that inequality \eqref{modulofs} directly implies that, for any $q\in \B$,
\[|f^s(q)|\le \sup_{q\in\B}|f(q)|^2=||f||^2_{\infty}.\]
\end{proof}

The H\"older inequality leads to the following result concerning the $*$-product of regular functions.
\begin{pro}\label{holder}
Let $p,q\in [1, +\infty]$ be conjugate exponents $\frac1p + \frac 1q =1$. If $f\in H^p(\B)$ and $g \in H^q(\B)$, then $f*g\in H^1(\B)$ and $g*f\in H^1(\B)$ .
\end{pro}
\begin{proof} If $f\equiv 0$ (or $g\equiv 0$) there is nothing to prove. Otherwise, let us consider the case in which $p,q\in (1, +\infty)$.
Fix $I\in \s$. For all $r\in [0,1)$ and $\theta \in [-\pi, \pi)$ such that $f(re^{I\theta})\neq 0$, set $$J(r, \theta)=(f(re^{I\theta}))^{-1}If(re^{I\theta}) \in \s$$ so that recalling Proposition \ref{trasf} we can write
\begin{equation}\label{eq1}
f*g(re^{I\theta})=f(re^{I\theta})g(re^{J(r,\theta)\theta}).
\end{equation}
The function $g$ can be represented as
$$
g(re^{J(r,\theta)\theta})=\frac12(g(re^{I\theta})+g(re^{-I\theta}))+\frac{J(r, \theta)I}{2}(g(re^{-I\theta})-g(re^{I\theta}))
$$
and hence, using the convexity of the function $x \mapsto x^q$ for $x\geq 0$ (since $q>1$), 
\begin{equation}\label{eq2}
|g(re^{J(r,\theta)\theta})|^q\leq 2^{q-1}(|g(re^{I\theta})|^q+|g(re^{-I\theta})|^q).
\end{equation}
Now, using the H\"older inequality and equations \eqref{eq1}, \eqref{eq2} we obtain
\begin{equation}
\begin{aligned}
\frac{1}{2\pi}\int_{-\pi}^{\pi}|f*g(re^{I\theta})|d\theta &
\leq \Big( \frac{1}{2\pi}\int_{-\pi}^{\pi}|f(re^{I\theta})|^pd\theta\Big)^{\frac1p} \Big( \frac{1}{2\pi}\int_{-\pi}^{\pi}|g(re^{J(r,\theta)\theta})|^qd\theta\Big)^{\frac1q} \\
&\le \Big( \frac{1}{2\pi}\int_{-\pi}^{\pi}|f(re^{I\theta})|^pd\theta\Big)^{\frac1p} \Big( \frac{2^q}{2\pi}\int_{-\pi}^{\pi}|g(re^{I\theta})|^qd\theta\Big)^{\frac1q},
\end{aligned}
\end{equation}
which leads to
$$
||f*g||_1\le 2||f||_p||g||_q < +\infty.
$$
For the remaining cases  ($p=1, q=+\infty$ and $p=+\infty, q=1$) the proofs follow the same lines.
\end{proof}

\begin{coro}\label{holder11}
Let $p\in [2, +\infty]$. If $f$ and $g$ belong to $H^p(\B)$, then $f*g\in H^1(\B)$ and $g*f\in H^1(\B)$ .
\end{coro}
\begin{proof}
If $\frac1p + \frac 1q =1$ and $p\ge2$, then $1\le q\le 2$. Therefore $H^p(\B) \subset H^q(\B)$ and hence $g$ belongs to $H^q(\B)$. Theorem \ref{holder} leads to the conclusion.
\end{proof}

\section{Boundary values of regular functions}\label{boundary}

A very important result, that in the classic case is quite laborious to reach, states that all functions in $H^p(\B)$ have radial limit along almost any ray.
\begin{pro}\label{radiale}
Let $f\in H^p(\B)$ for some $p\in(0,+\infty]$. Then for any $I\in \s$, the limit
\[\lim_{r\to 1^-}f(re^{I\theta})=\tilde f(e^{I\theta})\]
exists for almost every $\theta \in [0, 2\pi)$.
\end{pro}
\begin{proof} Let $I\in \s$. Write the splitting of $f$ on $\B_I$ as $f_I(z)=F(z)+G(z)J$, then Proposition \ref{FGHp} yields that the holomorphic functions $F$ and $G$ are both in $H^p(\B_I)$.
Classical results in the theory of $H^p$ spaces (see e.g. \cite{Hoffman}) yield that
the radial limits
\[ \lim_{r\to1^{-}}F(re^{I\theta})=\widetilde F(e^{I\theta}) \quad \text{and} \quad \lim_{r\to1^{-}}G(re^{I\phi})=\widetilde G(e^{I\phi})\]
exist respectively for almost every $\theta$ and for almost every $\phi$. 
Therefore the radial limit
\[\lim_{r\to 1^-}f(re^{I\theta})=\lim_{r\to 1^-}(F(re^{I\theta})+G(re^{I\theta})J)=\widetilde F(re^{I\theta})+\widetilde G(re^{I\theta})J=\tilde f(e^{I\theta})\]
exists for almost every $\theta \in [0,2\pi)$.
\end{proof}
\begin{oss} The previous result is slightly stronger than its complex counterpart: in fact, if $f\in H^p(\B)$ for some $p \in (0, +\infty]$, then on {\em each} slice, the radial limit of $f$ exists along almost any ray. 
\end{oss}

From now on, we will denote by $\tilde{f}$ the \emph{radial limit} of a function $f\in  H^p(\B)$. As it happens in the complex case, the function  $\tilde f$ is measurable on $\p \B$.

\begin{pro}\label{quasiognisfera}
Let $f\in H^p(\B)$ for some $p\in(0,+\infty]$. Then for almost every $\theta\in [-\pi, \pi)$, the limit
\[\lim_{r\to 1^-}f(re^{I\theta})=\tilde f(e^{I\theta})\]
exists for every $I \in \s$. Namely, the radial limit of $f$ exists at all points of the sphere $\cos \theta +(\sin \theta ) \s$ (contained in the boundary of $\B$) for almost every $\theta \in [-\pi, \pi)$.
\end{pro} 
\begin{proof}Choose $I,J\in \s$ and use the Representation Formula for $f$
\begin{equation*}
f(re^{J\theta})=\frac{1}{2}\left[ f(re^{I\theta})+f(re^{-I\theta})\right]+\frac{JI}{2}\left[f(re^{-I\theta})-f(re^{I\theta}) \right].
\end{equation*}
Since for almost every $\theta \in [-\pi, \pi)$ the radial limit $\tilde f$ exists both at $e^{I\theta}$ and at  $e^{-I\theta}$, then formula
\begin{equation*}
\tilde f(e^{J\theta})=\frac{1}{2}\left[ \tilde f(e^{I\theta})+\tilde f(e^{-I\theta})\right]+\frac{JI}{2}\left[\tilde f(e^{-I\theta})-\tilde f(e^{I\theta}) \right].
\end{equation*}
leads to the conclusion.
\end{proof}

Moreover, we will show that radial limits of a (non identically zero) function in $H^p(\B)$ can not vanish on a subset of positive measure of the boundary of the ball.  
\begin{pro}\label{limn0}
Let $f\in H^p(\B)$ for some $p\in(0,+\infty]$, $f\not \equiv 0$. Then, for any $I\in\s$, for almost every $\theta \in [-\pi,\pi)$,
\[\lim_{r\to 1^-}f(re^{I\theta})=\tilde f(e^{I\theta}) \neq 0.\]  
\end{pro}
\begin{proof}
Choose any $J \in \s$ orthogonal to $I$. If $f$ splits on $\B_I$ as
$f_I(z)=F(z)+G(z)J,$
then the splitting components $F$ and $G$ are in $H^p(\B_I)$ and (thanks to the Identity Principle for regular functions, see \cite{libroGSS}) at least one of them is not identically vanishing. 
Suppose that $F \not \equiv 0$ on $\B_I$. The classical result stated, e.g., in Theorem 17.18, \cite{Rudin}, yields that 
$\lim_{r\to1^-}F(re^{I\theta})=\widetilde{F}(e^{I\theta})\neq 0$ for almost every $\theta \in [-\pi,\pi)$.
Thanks to the orthogonality of $I$ and $J$, we easily conclude that for almost every $\theta\in[-\pi,\pi)$,
\[\tilde{f}(e^{I\theta})=\widetilde{F}(e^{I\theta})+\widetilde{G}(e^{I\theta})J\neq 0.\]
\end{proof}

This easy consequence of the previous result will be used in the sequel.

\begin{oss} \label{conssegno}If $f\in H^p(\B)$ for some $p\in(0,+\infty]$, $f\not \equiv 0$, we have that  for any $I\in\s$ and for almost every $\theta \in [-\pi,\pi)$, there exists $r_0>0$ such that  $f(re^{I\theta})\neq 0$  for all $r\in[r_0, 1)$.

\end{oss}

Let us define the $*$-product and the $*$-inverse for radial limits. To this aim we prove the following statement.

\begin{pro}\label{prodlim}
Let $f\in H^p(\B)$, $f\not \equiv 0$, and $g\in H^q(\B)$ for some $p,q\in(0,+\infty]$, and let $\tilde{f}$ and $\tilde{g}$ be their (almost everywhere) radial limits. 
For any $I\in \s$, 
for almost every $\theta \in [-\pi,\pi)$,
\[\lim_{r\to1^-}f*g(re^{I\theta})=\tilde{f}(e^{I\theta})\tilde{g}(\tilde{f}(e^{I\theta})^{-1}e^{I\theta}\tilde{f}(e^{I\theta})).\]
Moreover, for any $I\in \s$, we have that $\lim_{r\to1^-}f^{-*}*g(re^{I\theta})$ exists (possibly infinite) for almost every  $\theta\in [-\pi, \pi)$, and when finite,
\[\lim_{r\to1^-}f^{-*}*g(re^{I\theta})=\tilde{f}(\tilde{f^c}(e^{I\theta})^{-1}e^{I\theta}\tilde{f^c}(e^{I\theta}))^{-1}\tilde{g}(\tilde{f^c}(e^{I\theta})^{-1}e^{I\theta}\tilde{f^c}(e^{I\theta})).\]
\end{pro}
\begin{proof}
Proposition \ref{limn0} yields that if $f$ is not vanishing identically, then for any $I\in \s$, $\tilde f(e^{I\theta})\neq 0$ for almost every $\theta \in [-\pi,\pi)$, and the same holds for $f^c$. 
Set 
\[T(q)=f(q)^{-1}qf(q).\]
Thanks to Remark \ref{conssegno},  $T(re^{I\theta})=f(re^{I\theta})^{-1}re^{I\theta}f(re^{I\theta})$ is well defined for any $I\in\s$, for almost every $\theta\in [-\pi,\pi)$, and for $1-r$ sufficiently small. Moreover the radial limit 
\[\lim_{r\to1^-}T(re^{I\theta})=\tilde{f}(e^{I\theta})^{-1}e^{I\theta}\tilde{f}(e^{I\theta})=:\widetilde{T}(e^{I\theta})\]
exists for any $I\in\s$ and almost every $\theta\in[-\pi,\pi)$.
Given any $I \in \s$, if $f(re^{I\theta})\neq0$ set $J(r,\theta)=f(re^{I\theta})^{-1}If(re^{I\theta})\in\s$; then, for almost every $\theta$ the radial limit 
\[\lim_{r\to 1^-}J(r,\theta)=\tilde{J}(\theta)\]
exists and belongs to $\s$. Hence we can write 
\[\tilde T(e^{I\theta})=\lim_{r\to1^-}T(re^{I\theta})=\lim_{r\to1^-}\big(r\cos \theta+r\sin \theta(f(re^{I\theta})^{-1}If(re^{I\theta})\big)=\cos\theta+(\sin\theta)\tilde{J}(\theta)=e^{\tilde J(\theta)\theta}.\] 
Using the Representation Formula \ref{RF} twice, for almost every $\theta$ we can write
\begin{equation*}
\begin{aligned}
\lim_{r\to 1^-}g(T(re^{I\theta}))&=\lim_{r\to 1^-}(g(r
e^{J(r,\theta)\theta}))\\&=\lim_{r\to 1^-}\bigg(\frac 12\big(g(re^{I\theta})+g(re^{-I\theta}) \big)+\frac{J(r,\theta)I}{2}\big(g(re^{-I\theta})-g(re^{I\theta})\big)\bigg)\\ 
&=\frac 12\big(\tilde{g}(e^{I\theta})+\tilde{g}(e^{-I\theta}) \big)+
\frac{\tilde{J}(\theta)I}{2}\big(\tilde{g}(e^{-I\theta})-\tilde{g}(e^{I\theta})\big)\\ 
&=\lim_{r\to 1^-} \left(\frac 12\big({g}(re^{I\theta})+{g}(re^{-I\theta}) \big)+
\frac{\tilde{J}(\theta)I}{2}\big({g}(re^{-I\theta})-{g}(re^{I\theta})\big)\right)\\ 
&= \lim_{r\to 1^-} g(re^{\tilde J (\theta) \theta})=\tilde{g}(e^{\tilde{J}(\theta)\theta})=\tilde{g}(\widetilde{T}(e^{I\theta})).
\end{aligned}
 \end{equation*}
Hence, recalling Proposition \ref{trasf}, we have
\[\lim_{r\to1^-}f*g(re^{I\theta})=\lim_{r\to1^-}f(re^{I\theta})g(T(re^{I\theta}))=\tilde{f}(e^{I\theta})\tilde{g}(\widetilde{T}(e^{I\theta})).\]
The same arguments apply also to the proof for the regular quotient.
\end{proof}
We are now ready to give the announced definitions. 
\begin{defn}
Let $f\in H^p(\B)$ and $g\in H^q(\B)$ for some $p,q\in(0,+\infty]$.
For any $I\in \s$, for almost every $\theta\in [-\pi,\pi)$, let  
$\tilde{f}(e^{I\theta})$ and $\tilde{g}(e^{I\theta})$
be the radial limits of $f$ and $g$.
We define the $*$-product of $\tilde{f}$ and $\tilde{g}$ as
\begin{equation*}
\begin{aligned}
\tilde{f}*\tilde{g}(e^{I\theta})=\lim_{r\to1^-}f*g(re^{I\theta})\\
\end{aligned}
\end{equation*}
for almost every $\theta$.
If moreover $\tilde f (e^{I\theta}) \neq 0$ at all points $e^{I\theta}$ where it is defined, then
we can define the $*$-quotient of $\tilde{f}$ and $\tilde g$  as
\[\tilde{f}^{-*}*\tilde{g}(e^{I\theta})=\lim_{r\to 1^-}f^{-*}*g(re^{I\theta})\]
for almost every $\theta$. In particular, if $g\equiv 1$, we obtain the definition of the $*$-inverse of $\tilde f$.
\end{defn}
Thanks to the existence of the radial limit it is possible to obtain integral representations for functions in $H^p(\B)$ for $p\in [1, +\infty]$.
\begin{teo}\label{PoissonCauchy}
If $f\in H^p(\B)$ for $p\in [1, +\infty]$, then, for any $I\in\s$, $f_I$ is the Poisson integral and the Cauchy integral of 
its radial limit $\tilde f_I$, i.e.,
\[
f_I(re^{I\theta})= \frac{1}{2\pi}\int_{-\pi}^{\pi} \frac{1-r^2}{1-2r\cos (\theta -t) +r^2}\tilde{f}_I(e^{It})dt
\]
and
\[
f_I(z)=\frac{1}{2\pi I}\int_{\p \B_I} \frac{d\zeta}{\zeta- z}\tilde{f}_I(\zeta).
\]
\end{teo}
\begin{proof} The proof is an application of the corresponding results for holomorphic functions to the splitting components of $f_I$.
\end{proof}
%
Our next goal is to show that, for any $p\in (0,+\infty]$, the radial limits $\tilde f_I$ of the restrictions of a function $f$ in $H^p(\B)$, are $L^p$ functions on the circle $\p \B_I$.
\begin{defn}
Let $g$ be a (quaternion valued) function defined (almost everywhere) on $\p \B_I$ and such that $|g|$ is measurable.  If $p\in(0,+\infty)$ we set $||g||_{L^p}$ to be the integral mean
\[||g||_{L^p}=\left(\frac{1}{2\pi}\int_{-\pi}^{\pi}|g(e^{I\theta})|^pd\theta\right)^{\frac{1}{p}},\]
If $p=+\infty$, we set 
\[||g||_{L^{\infty}}=\ess \!\sup_{\theta\in(-\pi,\pi)}|g(e^{I\theta})|.\] 
For any $p\in(0,+\infty]$, we denote by $L^p(\partial\B_I)$ the standard $L^p$ space,
\[ L^p(\partial\B_I)=\{g:\p \B_I\to \HH \, \ | \, \ |g| \ \ \text{is measurable and}\ \ ||g||_{L^p}<+\infty \}.\]  
\end{defn}
\noindent Let us now point out that,  if
$(f_I)_r(e^{I\theta})=f(re^{I\theta})$, then
\[||(f_I)_r||_{L^p}=\left(\frac{1}{2\pi}\int_{-\pi}^{\pi}|f(re^{I\theta})|^pd\theta\right)^{\frac{1}{p}}=M_p(f_I,r) \]
for all $r\in[0,1)$ and all $I\in \s$.
\begin{pro}\label{convLp}
Let $f\in H^p(\B)$ for some $p\in(0,+\infty)$. Then, for any $I\in\s$, $\tilde f_I-(f_I)_r$ belongs to $L^p(\p \B)$ and
\[\lim_{r\to 1^-}||\tilde{f}_I- (f_I)_r||_{L^p}=0.\]
\end{pro}
\begin{proof} An application of the analogous result in the complex case (seeTheorem 2.6 in \cite{duren}) to the splitting components of $f_I$ leads to the conclusion.
\end{proof}

\noindent Now we are able to prove the desired result.
\begin{pro}\label{p=lp}
Let $f\in H^p(\B)$ for some $p\in(0+\infty]$. Then, for any $I\in\s$, the function $\tilde f_I: \p\B_I \to \HH $ defined for almost every $\theta\in[0,2\pi)$, by
\[\tilde{f}(e^{I\theta})=\lim_{r\to1^-}f(re^{I\theta}),\]
does belong to $L^p(\p \B_I)$ and
\[||\tilde{f}_I||_{L^p}=||f_I||_p.\]
\end{pro}
\begin{proof} The proof is a direct consequence of 
Proposition \ref{convLp}
and of the Poisson integral representation stated in Theorem \ref{PoissonCauchy}.
\end{proof}

\section{Factorization theorems}\label{factorization}
In the classical setting it is possible to decompose a holomorphic function in $H^p(\D)$ into its {\em inner} and {\em outer} factors, 
see Chapter 5 of \cite{Hoffman}. The quaternionic counterparts are defined as follows.

\begin{defn}\label{outer}
A regular function $E\in H^1(\B)$ is an {\em outer function} if for any $f\in H^1(\B)$ such that $|\widetilde{E}(q)|=|\tilde{f}(q)|$ 
for almost any $q\in\p \B$, we have
\[|E(q)|\ge|f(q)| \quad \text{for any $q\in\B$}.\]
\end{defn}
\noindent In the complex setting, the definition of outer function can be given equivalently in terms of a never-vanishing holomorphic function, 
expressed by means of 
an integral representation 
(see \cite{Hoffman}, Chapter 5). 
This correspondence fails to be true for regular functions, since in general we can not reproduce the same construction.

\begin{defn}\label{inner}
A regular function $\mathcal{I}\in H^{\infty}(\B)$ is an {\em inner function} if $|\mathcal{I}(q)|\le 1$ 
for any $q\in \B$ and $|\widetilde{\mathcal{I}}(q)|=1$ for almost any $q\in \p \B$.
\end{defn}

\noindent We recall that, in the complex setting, each inner function can be factored into a product of two distinct types of inner functions, 
namely a {\em Blaschke product} and a {\em singular function}.
Let us define the quaternionic analogues of singular functions, and then study the analogues of 
Blaschke products. 
\begin{defn}
An inner function $f\in H^{\infty}(\B)$ is a {\em singular function} if $f$ is non-vanishing on $\B$.
\end{defn} 

In the complex setting there are two possible approaches to the factorization of an $H^p$ function (compare \cite{duren, Hoffman}). 
The first one is to begin by the extraction of the outer factor, thus obtaining the inner one. 
At this point, extracting the zeros one separates the Blaschke product and the singular part. 
The other possibility is to start with the extraction of the zeros, thus identifying the Blaschke product, 
and then separate the outer factor from the singular one.     
In the quaternionic setting, since we can not reproduce the construction of the outer factor, 
let us begin with the extraction of the zeros of a function $f$ in $H^p(\B)$ for $p\in(0,+\infty]$. \\
Thanks to the characterization of the zero set of regular functions (see Theorem \ref{zeri}), recalling  Definitions \ref{generatori} and \ref{multiplicity}, 
we can build a sequence representing the zeros of $f$.

\begin{defn}
Let $f$ be a regular function.
The {\em sequence of zeros} of $f$ is a sequence $\{a_n\}_{n\in\N}$, contained in the zero set of $f$, composed as follows: 
the isolated zeros are listed according to their isolated multiplicity; the spherical zeros are represented by any element that generates 
the $2$-sphere of zeros together with its conjugate, listed according to their spherical multiplicity. 
Namely, if $a_\ell$ generates a spherical zero (not containing $a_{\ell-1}$) with spherical multiplicity $2m$, 
then $a_{\ell+2k}=a_\ell$ and $a_{\ell+2k+1}=\overline{a}_\ell$ for all $k=0,\dots,m-1$.
\end{defn}
\noindent In analogy with the complex case, we can give the following two definitions (previously introduced in \cite{milanesi}).
\begin{defn}
Let $a$ be a point in $\B$. The {\em Blaschke factor} associated with $a$ is the regular Moebius transformation defined as
\[M_a(q)=(1-q\overline{a})^{-*}*(a-q)\frac{\overline{a}}{|a|}
\text{\hskip .3cm if $a\neq 0$, \hskip .2cm and\hskip .3cm}  M_a(q)=q
\text{\hskip .3cm if $a=0$}.\]
\end{defn}
\begin{defn}
If $\{a_n\}_{n\in\N}$ is a sequence of points in $\B$ such that the infinite regular product 
\[B(q)=\stella_{n\ge 0}M_{a_n}(q)\]
converges uniformly on compact subsets of $\B$, then $B$ is called {\em Blaschke product}, and it defines a regular function on $\B$ 
(see 
\cite{libroGSS}). 
\end{defn}
Quaternionic Blaschke products are also treated in \cite{3,4,2}. 
Blaschke products are examples of inner functions. 
\begin{teo}
Let 
\[B(q)=\stella_{n\ge 0}M_{a_n}(q)\] 
be a Blaschke product. 
Then $|B(q)|\le 1$ for any $q\in \B$ and $|\widetilde{B}(q)|=1$ for almost any $q\in\p\B$.
\end{teo}
\begin{proof}
For $k\in \N$ consider the finite (regular) product $B_k$,
\[B_k(q)=\stella_{n=0}^kM_{a_n}(q)=\prod_{n=0}^kM_{a_n}(T_n(q)),\]
where  $T_0(q)=q$ and  $T_\ell$ is defined iteratively (outside the zero set of $B_{\ell-1}$), in view of Proposition \ref{trasf}, by
\[T_\ell(q)=\left(\stella_{j=0}^{\ell-1}M_{a_j}(q)\right)^{-1}q\stella_{j=0}^{\ell-1}M_{a_j}(q)=\left(\prod_{j=0}^{\ell-1}M_{a_j}(T_j(q))\right)^{-1}q\prod_{j=0}^{\ell-1}M_{a_j}(T_j(q)).\]
for any $\ell=1,\dots, k$. Since $|T_n(q)|=|q|<1$ for any $q\in \B$ and since each factor $M_{a_n}$ is bounded in modulus by $1$ on $\B$ (see Proposition \ref{moeb2}), 
we get that we can bound each finite product,
\[\left|B_k(q)\right|=\prod_{n=0}^k\left|M_{a_n}(T_n(q))\right|<1.\]
Thanks to the uniform convergence on compact subsets of the finite products to $B$ we get that
\[|B(q)|\leq 1 \quad \text{for any $q\in \B$}.\]
Hence $B\in H^{\infty}(\B)$ and therefore for any $I\in\s$, for almost any $\theta\in [-\pi,\pi)$ there exists the radial limit 
\[\lim_{r\to 1^-}B(re^{I\theta})=\widetilde{B}(e^{I\theta}).\]
The same clearly holds true for any finite product $B_k\in H^{\infty}(\B)$.
For any $k \in \N$, the following regular function
\[B_k^{-*}*B(q)=\stella_{n\ge k+1}M_{a_n}(q)\]
is a Blaschke product (and hence a bounded regular function) and 
\[\lim_{k \to \infty}B_k^{-*}*B(q)=1\]
uniformly on compact subsets of $\B$.
Observe that, for any $k\in \N$, the finite regular product $B_k$ (as well as its regular conjugate $B^c_k$) is regular up to the closure of $\B$.
Moreover, both $B_k$ and $B^c_k$ only have finitely many zeros in the interior of $\B$. Hence (see Proposition \ref{Caterina}) the function
 \[\tau_k(q)=\big(B_k^c(q)\big)^{-1}qB_k^c(q)\]
 is a diffeomorphism of a neighborhood (for instance a spherical shell) of $\p \B$ onto itself,
  it maps the boundary of $\B$ onto itself and it has inverse given by
 \[\tau^{-1}_k(q)=(B_k(q))^{-1}qB_k(q).\]
%
Let $I\in\s$. Proposition \ref{Caterina} yields that we can write
\begin{equation*}
\begin{aligned}
 \frac{1}{2\pi}\int_{-\pi}^{\pi}\big|B_k^{-*}*B(\tau^{-1}_k(re^{I\theta}))\big|d\theta&= 
\frac{1}{2\pi}\int_{-\pi}^{\pi}\big|{B}_k(re^{I\theta})\big|^{-1}\big|B(re^{I\theta})\big|d\theta\\
&\le \frac{1}{2\pi}\int_{-\pi}^{\pi}\big( \max_{\theta\in [-\pi,\pi]}\big|{B}_k(re^{I\theta})\big|^{-1}\big)\big|B(re^{I\theta})\big|d\theta.
\end{aligned}
\end{equation*}
Since $B_k$ is a finite Blaschke product, it maps $\p\B$ to itself and 
\[\lim_{r\to 1^-}\big|B_k(re^{I\theta}) \big|=\big|\widetilde{B}_k(e^{I\theta})\big|=1,\]
for any $\theta$. 
Hence, for any $\varepsilon>0$ there exists $r(\varepsilon)$ such that for any $r(\varepsilon)\le r<1$
\[\max_{\theta\in [-\pi,\pi]}\big|{B}_k(re^{I\theta})\big|^{-1}\le 1+\varepsilon.\]
Therefore, for any $\varepsilon>0$ there exists $r$ sufficiently close to $1$ such that
\begin{equation}\label{disu}
\begin{aligned}
&\frac{1}{2\pi}\int_{-\pi}^{\pi}\big|B_k^{-*}*B(\tau^{-1}_k(re^{I\theta}))\big|d\theta 
\le \frac{1+\varepsilon}{2\pi}\int_{-\pi}^{\pi}\big|B(re^{I\theta})\big|d\theta \le \frac{1+\varepsilon}{2\pi}\int_{-\pi}^{\pi}\big|\tilde{B}(e^{I\theta})\big|d\theta\le 1+\varepsilon
\end{aligned}
\end{equation}
where we use Proposition \ref{incr} and the fact that $||\widetilde B||_{L^\infty}=||B||_{\infty}\le 1$.   
Set
\[ J_k(r,\theta)= \big(B_k(re^{I\theta})\big)^{-1}IB_k(re^{I\theta}) \in \s, \]
so that
\[\tau^{-1}_k(re^{I\theta})=r\cos\theta+(r\sin \theta )J_k(r,\theta).\]
Using the Representation Formula \ref{RF} we can then write
\begin{equation*}
\begin{aligned}
\big|B_k^{-*}&*B(\tau^{-1}_k(re^{I\theta}))\big|=\big|B_k^{-*}*B(re^{J_{k}(r,\theta)\theta})\big|\\
&=\bigg|\frac 12 \big(B_k^{-*}\!*\!B(re^{I\theta})\!+\!B_k^{-*}\!*\!B(re^{-I\theta})\big)+
\frac{J_k(r,\theta)I}{2}\big(B_k^{-*}\!*\!B(re^{-I\theta})\!-\!B_k^{-*}\!*\!B(re^{I\theta})\big)\bigg|.\\
\end{aligned}
\end{equation*}
Since $B_k$ is a diffeomorphism of (a neighborhood of) $\p \B$ onto itself, we get that for every $\theta$ the limit
\[\lim_{r\to 1^-}J_k(r,\theta)=\tilde{J}_k(\theta)\]  
exists. Hence, recalling Proposition \ref{prodlim} and \ref{p=lp}, applying twice the Representation Formula \ref{RF}, we get
\begin{equation*}
\begin{aligned}
&\lim_{r\to 1^-}\frac{1}{2\pi}\int_{-\pi}^{\pi}\big|B_k^{-*}*B(\tau^{-1}_k(re^{I\theta}))\big|d\theta\\
&=\frac{1}{2\pi}\!\int_{-\pi}^{\pi}\!\bigg|\frac 12 \big(\widetilde{B}_k^{-*}\!*\!\widetilde{B}(e^{I\theta})\!+
\!\widetilde{B}_k^{-*}\!*\!\widetilde{B}(e^{-I\theta})\big)+
\frac{\tilde{J}_k(\theta)I}{2}\big(\widetilde{B}_k^{-*}\!*\!\widetilde{B}(e^{-I\theta})\!-
\!\widetilde{B}_k^{-*}\!*\!\widetilde{B}(e^{I\theta})\big)\bigg|d\theta\\
&=\lim_{r\to 1^-}\frac{1}{2\pi}\int_{-\pi}^{\pi}\bigg|\frac 12 \big(B_k^{-*}*{B}(re^{I\theta})\!+
{B}_k^{-*}*{B}(re^{-I\theta})\big)\\
&\hskip 6.4 cm+
\frac{\tilde{J}_k(\theta)I}{2}\big({B}_k^{-*}*{B}(re^{-I\theta})-
{B}_k^{-*}*{B}(re^{I\theta})\big)\bigg|d\theta\\
&=\lim_{r\to 1^-}\frac{1}{2\pi}\!\int_{-\pi}^{\pi}\!\big|B_k^{-*}\!*\!{B}(re^{\tilde{J}_k(\theta)\theta})\big|d\theta=\frac{1}{2\pi}\int_{-\pi}^{\pi}\big|\widetilde{B}_k^{-*}*\widetilde{B}(\tilde{\tau}^{-1}_k(e^{I\theta}))\big|d\theta
\end{aligned}
\end{equation*}
where $\tilde{\tau}_k^{-1}$ is the radial limit of $\tau^{-1}_k$  (see Proposition \ref{prodlim}).
Recalling inequality \eqref{disu}, we get then that for any $\varepsilon>0$
\begin{equation*}
\begin{aligned}
\frac{1}{2\pi}\int_{-\pi}^{\pi}\big|\widetilde{B}_k^{-*}*\widetilde{B}(\tilde{\tau}^{-1}_k(e^{I\theta}))\big|d\theta
&=\lim_{r\to 1^-}\frac{1}{2\pi}\int_{-\pi}^{\pi}\big|B_k^{-*}*B(\tau^{-1}_k(re^{I\theta}))\big|d\theta\\
&\le \frac{1+\varepsilon}{2\pi}\int_{-\pi}^{\pi}\big|\tilde{B}(e^{I\theta})\big|d\theta\le 1+\varepsilon.
\end{aligned}
\end{equation*}
Now we want to take the limit for $k \to +\infty$ of the previous inequality.
For almost every $\theta$, the limit
\[\lim_{k\to +\infty}\tilde{\tau}^{-1}_k(e^{I\theta})
=\lim_{k\to +\infty}\big(\widetilde{B}_k(e^{I\theta})\big)^{-1}e^{I\theta}\widetilde{B}_k(e^{I\theta})
=\big(\widetilde{B}(e^{I\theta})\big)^{-1}e^{I\theta}\widetilde{B}(e^{I\theta})=\tilde{\tau}^{-1}(e^{I\theta})
\]
does exist and it coincides with 
\[\lim_{k\to +\infty}\tilde{\tau}^{-1}_k(e^{I\theta})=\lim_{k\to+\infty} (\cos\theta +(\sin\theta)\widetilde{J}_k(\theta))=\cos\theta +(\sin\theta)\widetilde{J}(\theta),\]
which implies that $\widetilde{J}_k(\theta)$ converges for almost every $\theta$.
Using again the Representation Formula \ref{RF}, we have then that
\begin{equation*}
\begin{aligned}
&\lim_{k\to +\infty}\big|\widetilde{B}_k^{-*}*\widetilde{B}(\tilde{\tau}^{-1}_k(e^{I\theta}))\big|\\
&=\lim_{k\to +\infty}\bigg|\frac 12 \big(\widetilde{B}_k^{-*}\!*\!\widetilde{B}(e^{I\theta})\!+
\!\widetilde{B}_k^{-*}\!*\!\widetilde{B}(e^{-I\theta})\big)+
\frac{\tilde{J}_k(\theta)I}{2}\big(\widetilde{B}_k^{-*}\!*\!\widetilde{B}(e^{-I\theta})\!-
\!\widetilde{B}_k^{-*}\!*\!\widetilde{B}(e^{I\theta})\big)\bigg|.\\
\end{aligned}
\end{equation*}
Recalling that $B_k^{-*}*B$ converges uniformly on compact sets to the function constantly equal to $1$, and that 
$\tilde{J}_k(\theta)$ converges, we obtain
\begin{equation*}
\begin{aligned}
&\lim_{k\to +\infty}\big|\widetilde{B}_k^{-*}*\widetilde{B}(\tilde{\tau}^{-1}_k(e^{I\theta}))\big|\equiv 1.
\end{aligned}
\end{equation*}
Therefore we get that, for any $\varepsilon>0$, 
\begin{equation*}
\begin{aligned}
1&=\lim_{k\to +\infty}\frac{1}{2\pi}\int_{-\pi}^{\pi}\big|\widetilde{B}_k^{-*}*\widetilde{B}(\tilde{\tau}^{-1}_k(e^{I\theta}))\big|d\theta\le \frac{1+\varepsilon}{2\pi}\int_{-\pi}^{\pi}\big|\tilde{B}(e^{I\theta})\big|d\theta\le 1+\varepsilon,
\end{aligned}
\end{equation*}
that finally implies
\[\big|\tilde{B}(e^{I\theta})\big|=1
\]
for almost every $\theta$.
\end{proof}

In order to show that the sequence of zeros of a function $f$ in $H^p(\B)$ is such that the Blaschke product associated with it is convergent, 
we will use classical results in the theory of complex $H^p$ spaces, that apply to the symmetrization of $f$. 

\begin{oss}\label{zerifs}
Recall that the symmetrization $f^s$ of a regular function $f$ behaves exactly 
as a holomorphic function on any slice $L_I$. Hence, if $f^s$ is in $H^p(\B)$ (and therefore in $H^p(\B_I)$) for some $p\in(0,+\infty]$, 
and $f^s \not \equiv 0$, classical results (see e.g. Theorem 15.23 in \cite{Rudin}) yield that,
if $\{a^I_n\}_{n\ge 0}$ is the sequence of zeros of $f^s$ in $\B_I$, listed according to their multiplicity, then the Blaschke condition 
\begin{equation*}
\sum_{n\ge 0}(1-|a_n^I|)<+\infty
\end{equation*}
is fulfilled.
\end{oss}

\noindent Consequently, 
\begin{pro}\label{Blacon}
Let $p\in (0, +\infty]$, $f \in H^{p}(\B)$, $f\not \equiv 0$ and let $\{b_n\}_{n\in \N}$ be its sequence of zeros. Then $\{b_n\}_{n\in \N}$ satisfies the Blaschke condition
\[\sum_{n\ge0}\left( 1-|b_n| \right)<+\infty.\]
\end{pro}
\begin{proof}
Let us consider the symmetrization of $f$, $f^s$. 
Thanks to Proposition \ref{fsbounded} and Remark \ref{zerifs}, for any $I\in\s$, if $\{a_n^I\}_{n\in \N}$ is the sequence of zeros of $f^s$ on $\B_I$, 
then it satisfies the Blaschke contidion. 
If $b_n=b_{1,n}+b_{2,n}I_n$, let us set $b_n^I=b_{1,n}+b_{2,n}I$ for all $n \in \N$. 
Then $|b_n|=|b_n^I|$ for all $n\in \N$, and $\{b_n^I\}_{n\in \N}\subseteq \{a_n^I\}_{n\in \N}$.
Therefore (recalling that $|a^I_n|<1$)
\[\sum_{n\ge0}\left( 1-|b_n| \right)=\sum_{n\ge0}\left( 1-|b^I_n| \right)\le\sum_{n\ge0}\left( 1-|a^I_n| \right)<+\infty.\]
\end{proof}

\noindent The previous result implies that the Blaschke product built 
from the zeros of a regular function $f\in H^p(\B)$ for some $p\in (0,+\infty]$ does converge uniformly on compact sets (compare with \cite{milanesi}).
\begin{pro}\label{Blaschkef}
Let $f$ be in $H^p(\B)$ for some $p \in (0,+\infty]$ and let $\{a_n\}_{n \ge 0}$ be its sequence of zeros. 
If $M_{a_n}(q)$ denotes the Blaschke factor associated with $a_n$, 
\[M_{a_n}(q)=(1-q\overline{a}_n)^{-*}*(a_n-q)\frac{\overline{a}_n}{|a_n|}\]
($M_{a_n}(q)=q$ if $a_n=0$), then the Blaschke product
\[B(q)=\stella_{n\ge 0}M_{a_n}(q)\]
converges uniformly on compact sets of $\B$.
Moreover, the function $B$ is regular on $\B$.
\end{pro}
\begin{proof} In \cite{libroGSS} the convergence of infinite quaternionic $*$- products of regular functions is presented in detail. In particular the convergence of  $B(q)$ is equivalent to the convergence of $\sum_{n\ge 0}|1-M_{a_n}(q)|$.  With this in mind, the proof can be found in \cite{milanesi}.  
\end{proof}

\noindent We point out that the convergence of $B(q)$ depends only on the moduli $|a_n|$, $n\in\N$. 
This means that we can build a Blaschke product $\widehat{B}(q)$ having exactly the same sequence of zeros $\{a_n\}_{n\in\N}$ 
of a given regular function $f\in H^p(\B)$ (see also \cite{milanesi}). In fact, in order to build such a Blaschke product  $\widehat{B}(q)$, 
we have to consider the product of Blaschke factors associated with suitable conjugates of the points $a_n$, 
lying on the same $2$-spheres $x_n+y_n\s$ generated by $a_n$, taking into account Proposition \ref{trasf}.
\begin{pro}\label{zeribla}
Let $\{a_n\}_{n\in\N}$ be the sequence of zeros of a regular function $f\in H^p(\B)$ for some $p\in(0,+\infty]$. 
Then there exists a Blaschke product $\widehat{B}(q)$ having the same sequence of zeros.
\end{pro}
\begin{proof}
We will give the proof in the case in which all the zeros (both isolated and spherical) have multiplicity $1$. In this case, we can assume, without loss of generality, that  $a_j\neq a_k$ for all $j,k\in \N$. 
Our aim is now  to build a sequence $\{{b}_n\}_{n\in\N}$, where each ${b}_n$ is a conjugate of $a_n$, 
such that the Blaschke product associated with it 
\[\widehat{B}(q)=\stella_{n\ge 0}M_{{b}_n}(q)\]
has $\{a_n\}_{n\in\N}$ as its sequence of zeros.
The convergence of the Blaschke product is guaranteed by Proposition \ref{Blaschkef}. Since the regular multiplication does not conjugate the zeros of the first function in the $*$-product,
the first term of the sequence will be equal to $a_0$, 
\[{b}_0=a_0.\]
For the second term, we need to find ${b}_1$ such that
\[M_{{b}_0}(q)*M_{{b}_1}(q)=M_{a_0}(q)*(1-q\overline{b}_1)^{-*}*({b}_1-q)\frac{\overline{{b}}_1}{|{b}_1|}\]
vanishes at $q=a_1$.
Notice that, for any $k\in\N$,
\[(1-q\overline{{b}}_k)^{-*}*({b}_k-q)=({b}_k-q)*(1-q\overline{{b}}_k)^{-*},\]
because 
\[({b}_k-q)*(1-q\overline{{b}}_k)=(1-q\overline{{b}}_k)*({b}_k-q).\]
Hence 
\[M_{{b}_0}(q)*M_{{b}_1}(q)=M_{ a_0}(q)*({b}_1-q)*(1-q\overline{{b}}_1)^{-*}\frac{\overline{{b}}_1}{|{b}_1|},\]
and, thanks to Proposition \ref{trasf}, we can write
\[M_{{b}_0}(q)*M_{{b}_1}(q)=\left(M_{a_0}(q)({b}_1-T_1(q))\right)*\left((1-q\overline{{b}}_1)^{-*}\frac{\overline{{b}}_1}{|{b}_1|}\right),\]
where 
\[T_1(q)=\left(M_{a_0}(q)\right)^{-1}qM_{a_0}(q).\]
Therefore, if we want that this product vanishes at $a_1$, we can set
\[{b}_1= T_1(a_1),\]
well defined since $a_1\neq a_0$.
We can iterate this process, setting, for any $n\ge 1$,
\[T_n(q)=\left(\stella_{k=0}^{n-1}M_{{b}_k}(q)\right)^{-1}q\left(\stella_{k=0}^{n-1}M_{{b}_k}(q)\right),\]
so that
\begin{equation*}
\begin{aligned}
\left(\stella_{k=0}^{n-1}M_{{b}_k}(q)\right)*M_{{b}_n}(q)&=\left(\stella_{k=0}^{n-1}M_{{b}_k}(q)\right)*({b}_n-q)*(1-q\overline{{b}}_n)^{-*}\frac{\overline{{b}}_n}{|{b}_n|}\\
&=\left(\left(\stella_{k=0}^{n-1}M_{{b}_k}(q)\right)({b}_n-T_n(q))\right)*\left((1-q\overline{{b}}_n)^{-*}\frac{\overline{{b}}_n}{|{b}_n|}\right)
\end{aligned}
\end{equation*}
Hence, if we want that
\[\left(\stella_{k=0}^{n-1}M_{{b}_k}(q)\right)*M_{{b}_n}(q)\]
vanishes at $q=a_n$, we have to set
\[{b}_{n}=T_{n}(a_n),\]
well defined since $a_n\neq a_{n-1}$. 
In the case in which some of the zeros of $f$ have multiplicities greater than $1$, the proof follows the same lines; one has only to take into account that:
\begin{enumerate}
\item to each isolated zero $a_j$ of $f$ of multiplicity $p>1$ there corresponds the (regular) power of a Blaschke factor $M_{{b}_j}^{*p}(q)$ which vanishes at $b_j$ with multiplicity $p$;
\item to each spherical zero of $f$ containing $a_k, \overline a_k$ and having multiplicity $p>1$ there corresponds the  slice preserving factor $(M_{{b}_k}*M_{\overline {{b}}_k}(q))^{*p}=(M^s_{{b}_k}(q))^p$.
\end{enumerate}
\end{proof}

\begin{oss}
Since we transform the zeros of $f$ by conjugation, all real and spherical zeros of $f$ are not modified by this process.
\end{oss}

We can now prove our first result in the direction of finding a factorization for functions in $H^p(\B)$.

\begin{teo}\label{singolare1}
Let $f\in H^p(\B)$ for some $p\in(0, +\infty]$. Then we can factor $f$ as
\[f(q)=h*g(q)\]
where $h$ and $g$ are regular functions on $\B$ such that $h(q)\neq 0$ for any $q\in\B$ and $g$ is  a Blaschke product.
\end{teo}

\begin{proof} As in the proof of  Proposition \ref{zeribla}, we give the proof for the case in which all zeros of $f$ have multiplicity $1$ and the sequence of zeros $\{a_n\}_{n\in\N}$ of $f$ is an injective sequence.
Let us split the sequence of zeros as
\[\{a_n\}_{n\in\N}=\{\alpha_n\}_{n\in\N} \cup \{\beta_n\}_{n\in\N}\] 
where $\{\alpha_n\}_{n\in\N}$ is the sequence corresponding to spherical zeros, while $\{\beta_n\}_{n\in\N}$ 
is the sequences of isolated ones. 
Since $\{\alpha_n\}_{n\in \N}$ is contained in the sequence of zeros of the function $f$, clearly
\[\sum_{n\ge 0}(1-|\alpha_n|)\le \sum_{n\ge 0}(1-|a_n|)<+\infty.\] 
Hence the Blaschke product associated with $\{\alpha_n\}_{n\in \N}$
\[B_{\alpha}(q)=\stella_{n\ge 0}M_{\alpha_n}(q)\]
converges uniformly on compact sets thus defining a regular function, vanishing exactly at the spherical zeros of $f$. 
The function $B_{\alpha}(q)$ is slice preserving, in fact it contains only factors of the type
$M_{\alpha_n}*M_{\overline{\alpha}_n}=M^s_{\alpha_n}.$
Let $f_{\beta}(q)$ be the function defined as
\[f_{\beta}(q)=B_{\alpha}^{-*}*f(q)=B_{\alpha}(q)^{-1}f(q),\]
so that we can write
\[f(q)=B_{\alpha}(q)f_{\beta}(q).\]
Since the poles of $B_{\alpha}^{-*}$ are spherical zeros of $f$, then $f_{\beta}$ is regular on $\B$ and its sequence of zeros coincides with $\{\beta_n\}_{n\in\N}$.
The idea is now to ``make spherical'' all the zeros of $f_{\beta}$. In order to do it,
we want to find a Blaschke product $B_{\overline{\beta}}$ such that
\[f_{\beta}*B_{\overline{\beta}}(q)\]
vanishes at all spheres generated by $\{\beta_n\}_{n\in\N}$, namely such that the sequence of zeros of $f_{\beta}*B_{\overline{\beta}}(q)$ is
\[\{\beta_n,\overline{\beta}_n\}_{n\in\N}.\]
We can build $B_{\overline{\beta}}$ following the lines of the proof of Proposition \ref{zeribla}. 
%
If 
\[B_{\overline{\beta}}(q)= \stella_{n\ge 0}M_{{\gamma}_n}(q),\]
we can define the sequence $\{{\gamma}_n\}_{n\in\N}$ iteratively as follows.
The first quaternion, ${\gamma}_0$, is such that
\begin{equation*}
\begin{aligned}
f_{\beta}*M_{{\gamma}_0}(q)&=f_{\beta}*\left({\gamma}_0-q\right)*\left(1-q\overline{{\gamma}}_0\right)^{-*}\frac{\overline{{\gamma}}_0}{|{\gamma}_0|}
%
\end{aligned}
\end{equation*}
vanishes at $q=\beta_0$ and at $q=\overline{\beta}_0$.
Since $f_{\beta}$ vanishes at $q=\beta_0$ and $f$ does not vanish at $q=\overline{\beta}_0$, if $T_0$ is defined as
\[T_0(q)=\left(f_{\beta}(q)\right)^{-1}qf_{\beta}(q),\]
then we can write
\begin{equation*}
\begin{aligned}
f_{\beta}*M_{{\gamma}_0}(q)&=\left(f_{\beta}(q)\left({\gamma}_0-T_0(q)\right)\right)*\left(\left(1-q\overline{{\gamma}}_0\right)^{-*}\frac{\overline{{\gamma}}_0}{|{\gamma}_0|}\right).
%
\end{aligned}
\end{equation*}
Hence, if we set
\[{\gamma}_0=T_0(\overline{\beta}_0),\]
we have that
\[f_{\beta}*M_{{\gamma}_0}(q)\]
vanishes both at $q=\beta_0$ and at $q=\overline{\beta}_0$.
As we have done in order to prove Proposition \ref{zeribla}, we can iterate the process, setting, for any $n\ge 1$,
\[T_n(q)=\left(f_{\beta}*\stella_{k=0}^{n-1}M_{{\gamma}_k}(q)\right)^{-1}q\left(f_{\beta}*\stella_{k=0}^{n-1}M_{{\gamma}_k}(q)\right)\]
so that
\begin{equation*}
\begin{aligned}
\left(f_{\beta}*\stella_{k=0}^{n-1}M_{{\gamma}_k}\right)*M_{{\gamma}_n}(q)&=\left(f_{\beta}*\stella_{k=0}^{n-1}M_{{\gamma}_k}\right)*\left({\gamma}_n-q\right)*\left(\left(1-q\overline{{\gamma}}_n\right)^{-*}\frac{\overline{{\gamma}}_n}{|{\gamma}_n|}\right)\\
&=\left(\left(f_{\beta}*\stella_{k=0}^{n-1}M_{{\gamma}_k}\right)\left({\gamma}_n-T_n(q)\right)\right)*\left(\left(1-q\overline{{\gamma}}_n\right)^{-*}\frac{\overline{{\gamma}}_n}{|{\gamma}_n|}\right).
\end{aligned}
\end{equation*}
Since $T_n$ is well defined on $\overline {\beta}_n$, if we set
\[{\gamma}_n=T_n(\overline{\beta}_n)\]
we get that
\[\left(f_{\beta}*\stella_{k=0}^{n-1}M_{{\gamma}_k}\right)*M_{{\gamma}_n}(q)\]
vanishes both at $q=\beta_n$ and at $q=\overline{\beta}_n$.
The convergence of the infinite product $B_{\overline{\beta}}(q)$ 
is guaranteed by the fact that it is the Blaschke product associated with the sequence $\{T_n(\overline{\beta}_n)\}_{n\in \N}$  where each element $T_n(\overline{\beta}_n)$ 
has the same modulus of $\beta_n$, and each $\beta_n$ is contained in the sequence of zeros of a function in $H^p(\B)$.
Hence 
\[f_{\beta}*B_{\overline{\beta}}(q)\]
is a regular function that has only spherical zeros, and its sequence of zeros is $\{\beta_n,\overline{\beta}_n\}_{n\in\N}$.
Therefore, if we set 
\[\widetilde B_{\beta}(q)=\stella_{n\ge 0}\left(M_{\beta_n}*M_{\overline{\beta}_n}\right)(q)=\stella_{n\ge 0}M^s_{\beta_n}(q),\]
then we can write
\begin{equation}\label{fattorizzazionef0}
 f_{\beta}*B_{\overline{\beta}}(q)=\widetilde B_{\beta}*h(q)=\widetilde B_{\beta}(q)h(q),
 \end{equation}
for some function $h$, never vanishing and regular on $\B$. 
To prove the regularity of $h$, it suffices to observe that since $\widetilde B_{\beta}$ has real coefficients, 
it is regular and it has exactly the same zeros of $f_{\beta}*B_{\overline{\beta}}(q)$, then the regular quotient
\[\widetilde B_{\beta}^{-*}*(f_{\beta}*B_{\overline{\beta}})(q)=(\widetilde B_{\beta}(q))^{-1}(f_{\beta}*B_{\overline{\beta}})(q)\]
is well defined (and regular) on the entire ball $\B$.
Consider the regular conjugate of $B_{\overline{\beta}}$, and $*$-multiply on the right by $B^c_{\overline{\beta}}$ all terms of equality 
\eqref{fattorizzazionef0}. We obtain
\[f_{\beta}*B_{\overline{\beta}}*B_{\overline{\beta}}^c(q)=\widetilde B_{\beta}(q)h*B_{\overline{\beta}}^c(q),\]
that can also be written as
\[B_{\overline{\beta}}^s(q)f_{\beta}(q)=\widetilde B_{\beta}(q)h*B_{\overline{\beta}}^c(q).\]
Now notice that $B_{\overline{\beta}}^s(q)=\widetilde B_{\beta}(q)$ because they both are 
Blaschke products associated to the same spherical zeros.
Therefore we infer
\[f_{\beta}(q)=h*B_{\overline{\beta}}^c(q),\]
that, for $f$, means
\[f(q)=B_{\alpha}(q)h*B_{\overline{\beta}}^c(q)=h*B_{\alpha}*B_{\overline{\beta}}^c(q).\]
Setting $g(q)=B_{\alpha}*B_{\overline{\beta}}^c(q)$ leads to the conclusion of the proof. 
\end{proof}

Once ``extracted''  the zeros of a function in $H^p(\B)$, we would like to identify its outer factor and its singular part. 

\begin{pro}
Let $f\in H^p(\B)$ for some $p\in[1,+\infty]$ be such that $f^{-*}\in H^q(\B)$ where $\frac1p +\frac1q=1$.  Then $f$ is an outer function.
\end{pro}
\begin{proof}
Let $g\in H^p(\B)$ be such that $|\tilde{g}|=|\tilde{f}|$ almost everywhere on $\p \B$.
The regular function $h=f^{-*}*g$ belongs to $H^1(\B)$ thanks to Proposition \ref{holder}. Therefore 
recalling Proposition \ref{prodlim}, we get that for almost every $\theta \in [-\pi,\pi)$, 
\begin{equation*}
\begin{aligned}
\lim_{r\to 1^-}|h(re^{I\theta})|&=\lim_{r\to 1^{-}}|f^{-*}*g(re^{I\theta})|\\
&=|\tilde{f}(\tilde{f^c}(e^{I\theta})^{-1}e^{I\theta}\tilde{f^c}(e^{I\theta}))|^{-1}|\tilde{g}(\tilde{f^c}(e^{I\theta})^{-1}e^{I\theta}\tilde{f^c}(e^{I\theta}))|.\\
\end{aligned}
\end{equation*}
Since $f$ and $\tilde g$ coincide almost everywhere at the boundary, we get that for almost every $\theta\in[-\pi,\pi)$,
\[\lim_{r\to 1^-}|h(re^{I\theta})|=1.\]
Therefore the Poisson integral representation formula guarantees that 
\[|h(q)|\le 1\]
for any $q\in\B$.
With the same notation of Proposition \ref{Caterina}, we obtain that for all $q\in \B$
$$
1\ge |f^{-*}*g(q)|=|f(T_f(q))|^{-1}|g(T_f(q))|.
$$
Since $f^{-*}\in H^q(\B)$ then $f$ has no zeros and $T_f$ is a diffeomorphism of $\B$, yielding
$$
|g(q)|\le|f(q)|
$$ 
for any $q\in \B$. We thus conclude that $f$ is an outer function.
\end{proof}

\begin{oss} In the case of a regular function $f: \B \to \HH$ that is continuous and non-vanishing up to the boundary of $\B$, the previous result implies that $f$ is an outer function (since $f$ and $f^{-*}$ both  belong to $H^{\infty}(\B)$).
\end{oss}



For regular functions that preserve a slice $L_I$, the factorization can be done in a stronger and more satisfactory fashion, that very much resembles its complex counterpart. 

\begin{teo}\label{EI}
Let $f\in H^p(\B)$ for some $p\in(0,+\infty]$, be such that $f$ maps $\B_I$ to $L_I$ for some $I\in \s$. Then we can factor $f$ as
\[f(q)=E*\mathcal{I}(q),\]
where $E$ is an outer function in $H^p(\B)$ such that $|\widetilde{E}|=|\tilde f|$ almost everywhere on the boundary $\p \B$ 
and $\mathcal{I}$ is a inner function.
\end{teo}

\begin{proof}
The restriction of $f$ to $\B_I$ is a (complex) holomorphic function $F_I:\B_I \to L_I$ mapping $z\mapsto f_I(z)$. 
Let us define the function $E_I\colon \B_I \to L_I$ to be
\[E_I(z)=\exp\left(\frac{1}{2\pi}\int_{-\pi}^{\pi}\frac{e^{I\theta}+z}{e^{I\theta}-z}\log|F_I(e^{I\theta})|d\theta\right), \]
namely the outer factor of $F_I$.
From factorization results in the complex setting, (see for instance \cite{Hoffman}), we know
that we can write
\[F_I(z)=E_I(z)\mathcal{I}_I(z)\]
where $\mathcal{I}_I(z)$ is the inner factor of $F_I$.
In particular, since both $E_I$ and $\mathcal{I}_I$ map $\B_I$ to $L_I$, we can also write
\[ F_I(z)=E_I(z)\mathcal{I}_I(z)=E_I(z)*\mathcal{I}_I(z).\]
Hence
\[f(q)=\ext(f_I)(q)=\ext(F_I)(q)=\ext(E_I*\mathcal{I}_I)(q)=\ext(E_I)*\ext(\mathcal{I}_I)(q)\]
where the last equality is due to the Identity Principle for regular functions, \cite{libroGSS}.
Let us set $E(q)=\ext(E_I)(q)$ and $\mathcal{I}(q)=\ext(\mathcal{I}_I)(q)$.
Since $E_I(z)\neq 0$ for all $z\in \B_I$, Proposition \ref{zerislice} yields
that also $E$ is never vanishing on $\B$.
To estimate the modulus of $\mathcal{I}$, recall that a function that maps the slice $L_I$ to itself has the following properties 
(see Proposition \ref{maxminslice})
\begin{equation}\label{Massimo}
\max_{J\in \s}|\mathcal{I}(x+yJ)|=\max\{|\mathcal{I}(x+yI)|,\, |\mathcal{I}(x-yI)|\}=\max\{|\mathcal{I}_I(x+yI)|,\, |\mathcal{I}_I(x-yI)|\}\end{equation}
and 
\[\min_{J\in \s}|\mathcal{I}(x+yJ)|=\min\{|\mathcal{I}(x+yI)|,\, |\mathcal{I}(x-yI)|\}=\min\{|\mathcal{I}_I(x+yI)|,\, |\mathcal{I}_I(x-yI)|\}\]
for all $x,y$ such that $x+yI\in \B_I$. 
By classical results, the inner function $\mathcal I_I$ is bounded in modulus by $1$ and its uniform norm equals $1$, therefore we get, 
thanks to \eqref{Massimo} 
\[||\mathcal{I}||_{\infty}=\sup_{q\in\B}|\mathcal{I}(q)|=\sup_{z\in\B_I}|\mathcal{I}_I(z)|=||\mathcal{I}_I||_{\infty}=1.\]
As a consequence
\begin{equation}\label{minorediuno}
|\mathcal I(q)| \le 1 
\end{equation}
for all $q\in \B$.
%
Moreover, since $|\widetilde{\mathcal{I}}_I|$ equals $1$ almost everywhere on the boundary $\p \B_I$, we have that, 
for almost every $x+yI$ such that $x^2+y^2=1$,
\[\max\sx\{|\widetilde{\mathcal{I}}_I(x+yI)|,\, |\widetilde{\mathcal{I}}_I(x-yI)|\dx\}=\min\sx\{|\widetilde{\mathcal{I}}_I(x+yI)|,\, |\widetilde{\mathcal{I}}_I(x-yI)|\dx\}=1.\]
Therefore, for almost every $x+yJ$ such that $x^2+y^2=1$,
\[\max_{J\in \s}\sx|\widetilde{\mathcal{I}}(x+yJ)\dx|=\min_{J\in\s}\sx|\widetilde{\mathcal{I}}(x+yJ)\dx|=1\]
namely for almost every $q\in\p \B$, $|\widetilde{\mathcal{I}}(q)|=1$. 
To obtain the wanted properties of the modulus of $E$, let us denote by $T$ the transformation
\[T(q)=(E^c(q))^{-1}q E^c(q).\]
Since $E^c$ is non-vanishing on $\B$, $T$ is a diffeomorphism of $\B$ with inverse
\[
T^{-1}(q)=(E(q))^{-1}q E(q).
\]
 Then, thanks to Proposition \ref{Caterina} and to inequality \eqref{minorediuno}, we can write 
\begin{equation}\label{Ef}
1\ge |\mathcal{I}(T^{-1}(q))|
=\left| E^{-*}*f(T^{-1}(q))\right|=\left|E^{-1}(q)f(q)\right|=\left|E(q)|^{-1}|f(q)\right|
\end{equation}
for any $q\in \B$. Hence 
\[|E(q)|\ge|f(q)| \quad \text{for any $q\in \B$}.\]
The study of the behavior of the modulus $|\widetilde{E}|$ at the boundary, requires some more effort.
First of all, classical results on outer functions (see e.g. \cite{Hoffman}) imply that $E_I\in H^p(\B_I)$. 
Then, thanks to Proposition \ref{f_Ihp} we get that $E\in H^p(\B)$. 
%
Therefore, for any $J\in\s$, the function $E$ has radial limit $\widetilde{E}$ for almost any $x+yJ \in \p \B_J$. 
Hence, also the transformation $T^{-1}(q)=(E(q))^{-1}q E(q)$ does. 
This allows us to prove that $\widetilde{T^{-1}}$ maps almost every $2$-sphere $x+y\s \subset \p \B$ one-to-one onto itself.
In fact, for almost every sphere $x+y\s \subset \p \B$ the function $\widetilde{E}$ is defined at all points of $x+y\s$ (see Proposition \ref{quasiognisfera}). 
For such a sphere $x+y\s$, thanks to the Representation Formula \ref{RF} and to the fact that $E$ preserves $L_I$, there exist $b,c\in L_I$ such that 
\[\widetilde{E}(x+yJ)=b+Jc \quad \text{for any $J\in\s$}.\]
Take $x+yK\in x+y\s$, with $K\neq I$. We are going to show that we can find the only $J\in \s$ such that
$\widetilde{T^{-1}}(x+yJ)=x+yK$. This is possible if and only if
\[\left(\widetilde{E}(x+yJ)\right)^{-1}q\widetilde{E}(x+yJ)=  (b+Jc)^{-1}(x+yJ)(b+Jc)   =x+yK.\]
Since $x,y$ are real numbers, this is equivalent to
\[J(b-cK)=c+bK.\]
Now, since $K\neq I$ and since $b,c\in L_I$, necessarily $b-cK\neq 0$. Hence 
\[J=(bK+c)(b-cK)^{-1}\]
solves our problem.
Thanks to this property of the map $T^{-1}$, 
we get that $1= |\widetilde {\mathcal I} (\widetilde {T^{-1}}(q))|$ for almost every $q\in \p \B$. The radial limit version of equation \eqref{Ef}, holding for almost every $q\in\p\B$,
\begin{equation*}
1=|\widetilde {\mathcal I} (\widetilde {T^{-1}}(q))| =|\widetilde{E}(q)|^{-1}|\tilde{f}(q)|
\end{equation*}
leads to the equality
\begin{equation}\label{massimale}
|\widetilde{E}(q)|=|\tilde{f}(q)|
\end{equation}  
for almost every $q\in \p \B$,
which leads also to $||E||_p=||\widetilde{E}||_{L^p}=||\tilde f||_{L^p}=||f||_p.$
We want to show now that $E$ is an outer function. To this aim, let $g\in H^p(\B)$ be such that 
\begin{equation}\label{E=g}
|\widetilde E(q)|=|\tilde g(q)|
\end{equation}
for almost all $q\in \p \B$. If we restrict $E$ and $g$ to $\B_I$, and recall the definition of $E_I$, we can write (see Proposition \ref{quasiognisfera})
\begin{equation}
\begin{aligned}
\log |E_I(re^{I\theta})|&=\frac1{2\pi}\int_{-\pi}^{\pi}P_r(\theta-t)\log |\tilde f_I(e^{It})|dt\\
&=\frac1{2\pi}\int_{-\pi}^{\pi}P_r(\theta-t)\log |\tilde g_I(e^{It})|dt \ge \log |g_I(re^{I\theta})|
\end{aligned}
\end{equation}
where $P_r(t)$ is the Poisson kernel and where last inequality is due to the subharmonicity of $\log |g_I|$ (see the proof of Proposition \ref{incr}). As a consequence we get that 
\[
1\ge |E_I(z)|^{-1}|g_I(z)|= |E^{-*}_I*g_I(z)| = |(E^{-*}*g)_I(z)|
\]
for all $z\in \B_I$. Since $E^{-*}*g$ is regular on $\B$, using Proposition \ref{f_Ihp} we obtain that $E^{-*}*g$ belongs to $H^{\infty}(\B)$.  Then the radial limit of $E^{-*}*g$ exists at almost every point of $\p\B$. Equation \eqref{E=g} and Proposition \ref{prodlim} guarantee that for every $J\in \s$
\[
\lim_{r\to 1^-} |E^{-*}*g(re^{J\theta})| = |\widetilde{E^{-*}*g}(e^{J\theta})|=1
\]
for almost every $\theta \in [-\pi, \pi)$. Hence, by Proposition \ref{p=lp}, 
\[
||E^{-*}*g ||_{\infty} = ||\widetilde{E^{-*}*g} ||_{L^{\infty}} = 1.
\]
Recalling now that  
$
T^{-1}(q)=(E(q))^{-1}q E(q)
$
is a diffeomorphism of $\B$, thanks to Proposition \ref{Caterina} we obtain
\[
1\ge \left| E^{-*}*g(T^{-1}(q))\right|=\left|E^{-1}(q)g(q)\right|=\left|E(q)|^{-1}|g(q)\right|
\]
i.e.,
\[
|E(q)|\ge |g(q)|
\]
for all $q\in \B$. This concludes the proof.
\end{proof}

\begin{coro}
Let $f\in H^p(\B)$ for some $p\in(0,+\infty]$, be such that $f$ maps $\B_I$ to $L_I$ for some $I\in \s$. Then there exist an outer function $E\in H^p(\B)$, a singular function $S\in H^{\infty}(\B)$ and a Blaschke product $B$ such that 
\[
f(q)= E*S*B(q)
\]
for all $q\in \B$.
\end{coro}
\begin{proof} Theorem \ref{EI} allows us to factor $f=E*\mathcal I$ as a $*$-product of an outer function $E$ and an inner function $\mathcal I$. Theorem \ref{singolare1} guarantees now the existence of a Blaschke product $B$ and a non-vanishing  function $S$ such that $\mathcal I = S*B$. Following the lines of the proof of Theorem \ref{EI} that led us to show that the outer factor of $f$ belongs to $H^{p}(\B)$, one can prove that $S$ belongs to $H^{\infty}(\B)$ and hence is a singular function.
\end{proof}
We point out that the Beurling-Lax type Theorem in \cite{milanesi}, and the Krein-Langer type  factorization theorem that appears in \cite{3,4}, are naturally connected to our factorization results.

\end{document}